\documentclass[12pt,a4paper]{article}
\usepackage{amssymb,amsmath,amsthm}
\usepackage{latexsym}

\newtheorem{theorem}{Theorem}
\newtheorem{lemma}{Lemma}
\newtheorem{proposition}{Proposition}
\newtheorem{corollary}{Corollary}
\newtheorem{remark}{Remark}
\numberwithin{equation}{section}
\numberwithin{theorem}{section}
\numberwithin{lemma}{section}
\numberwithin{proposition}{section}
\numberwithin{corollary}{section}
\numberwithin{remark}{section}
\usepackage{color}
\newcommand{\CI}{{\mathcal I}}
\newcommand{\bg}{\begin{equation}}
\newcommand{\ed}{\end{equation}}
\newcommand{\bga}{\begin{eqnarray}}
\newcommand{\eda}{\end{eqnarray}}

\begin{document}

\title{Stability of time-dependent Navier-Stokes flow \\
and algebraic energy decay}
\author{Toshiaki Hishida\thanks{
Partially supported by JSPS Grant-in-aid for Scientific Research 24540169} \\
Graduate School of Mathematics, Nagoya University \\
Nagoya 464-8602 Japan \\
\texttt{hishida@math.nagoya-u.ac.jp} \\
and \\
Maria E. Schonbek \\
Department of Mathematics \\
University of California Santa Cruz \\
Santa Cruz, California 95064 USA \\
\texttt{schonbek@ucsc.edu}}
\date{}
\maketitle

\begin{abstract}
Let $V=V(x,t)$ be a given time-dependent Navier-Stokes flow
of an incompressible viscous
fluid in the whole space $\mathbb R^n$ $(n=3,4)$.
 Assume such V to be  small in 
$L^\infty(0,\infty; L^{n,\infty})$,
where $L^{n,\infty}$ denotes the weak-$L^n$ space.
The energy stability of this basic flow $V$
with respect to any initial disturbance in $L^2_\sigma$
has been established by Karch, Pilarczyk and Schonbek \cite{KPS}.
In this paper we study,  under reasonable conditions, the algebraic rates of
energy decay of disturbances as $t\to\infty$.
\end{abstract}

\section{Introduction}
\label{intro}

We consider  the asymptotic stability of the Navier-Stokes flow 
for an incompressible viscous fluid,
 governed by the system
\begin{equation}
\partial_t w+w\cdot\nabla w=\Delta w-\nabla\pi+g, \qquad
\mbox{div $w$}=0,
\label{NS-0}
\end{equation}
where $w=(w_1(x,t),\cdots,w_n(x,t))$ and $\pi(x,t)$ respectively
stand for the unknown velocity and pressure of the fluid,
while $g=(g_1(x,t),\cdots,g_n(x,t))$
denotes a given external force.
The results  obtained here also  study the  algebraic rates 
of decay of energy ($L^2$ decay) of disturbance as $t\to\infty$.
The basic flow   considered  is  time-dependent 
with minimal assumptions.
In this case the theory has been less developed than for stationary basic flows.

Given  $V=V(x,t)$ a  time-dependent Navier-Stokes flow
in the whole space $\mathbb R^n$ $(n=3,4)$ 
we address stability questions for weak solutions around $V$ without
any smallness condition on the initial $L^2_\sigma$ disturbances.
As important examples of this basic flows $V$, we have  in mind:
forward self-similar solution, 
time-periodic (even almost periodic) solution and global mild
(eventually strong) solution of the Cauchy problem.
It is reasonable to assume that
$V\in L^\infty(0,\infty; L^{n,\infty})$
as well as its weak*-continuity with values in $L^{n,\infty}$,
where $L^{n,\infty}$ denotes the weak-$L^n$ space.
In order to cover the self-similar solution,
we adopt $L^{n,\infty}$ rather than  the usual $L^n$.

The energy stability of small $V$ in the  above mentioned class
has been recently established by
Karch, Pilarczyk and Schonbek \cite{KPS}.
However  they did not obtain an algebraic rate of decay.
It would be interesting to find how fast the disturbance decays
as $t\to\infty$ when the initial disturbance possesses for instance
better summability at space infinity.
Yamazaki \cite{Y} studied another sort of stability of small $V$
in essentially the same class as above with respect to
small initial disturbance taken from $L^{n,\infty}$
and derived the decay rate of disturbance in
$L^{r,\infty}$ for $r\in (n,\infty)$.
Concerning the energy stability of another kind of time-dependent flow,
Kozono \cite{Ko} established the stability of large weak solution
$V\in L^q(0,\infty; L^r)$ of the Serrin class
with respect to any initial disturbance in $L^2_\sigma$,
where $2/q+n/r=1$ and $2\leq q<\infty$ ($n<r\leq \infty$).
Indeed the basic flow is not assumed to be small,
but it covers neither self-similar solution nor
time-periodic solution.
There are some other results on the stability of time-dependent flow,
see the references cited in \cite{KPS}, \cite{Ko} and \cite{Y}.

If we impose both more decay at space infinity and better local regularity
on the basic flow $V$
(for example, $L^{n,\infty}$ is replaced by $L^q\cap L^r$ with some 
$q<n<r$), 
then the stability analysis becomes less difficult.
We will discuss the stability of $V$ being only in
the critical space such as
$L^{n,\infty}$ from the viewpoint of scaling transformation under which 
the Navier-Stokes system \eqref{NS-0} is invariant.
We refer to \cite{KPS} for several critical spaces.
Besides $L^{n,\infty}$ there are some other critical spaces which contain
homogeneous functions of degree $-1$, so that one can include the 
self-similar solution to the class of flows under consideration.
The reason why we  will work with solutions in  $L^{n,\infty}$  
is closely related to insight
due to Yamazaki \cite{Y}, that will be mentioned later.
We are mainly interested in the physically relevant case $n=3$
as in \cite{KPS}, but the argument for $n=4$ is completely parallel.
For $n\geq 5$
the  justification of the integral equation for weak solutions does not work,
see \eqref{integ},
but the results in this paper can be easily extended 
to such higher dimensional case
when assuming more regularity of the basic flow $V$, see Reamrk \ref{reg-2}.
Since we prefer results under minimal assumptions on $V$,
we restrict ourselves to the case $n=3,4 $.
For $n=2$, the summability $L^{2,\infty}$
is not enough at space infinity and thus one needs a different approach.

When $V=0$, the energy decay problem for 
\eqref{NS-0} (with $g=0$) was raised 
in his celebrated paper \cite{L} by Leray and, fifty years later, 
it was solved by Kato \cite{Ka} and by Masuda \cite{Ma}, independently.
In both papers \cite{Ka} and \cite{Ma}
it was clarified that the strong energy inequality
\begin{equation}
\begin{split}
&\|w(t)\|_2^2+2\int_s^t\|\nabla w\|_2^2\,d\tau\leq \|w(s)\|_2^2 \\
&\mbox{for a.e. $s\geq 0$, including $s=0$, and all $t\geq s$}
\end{split}
\label{SEI-0}
\end{equation}
plays an important role in deduction of energy decay,
where $\|\cdot\|_2$ denotes the $L^2$-norm.
Later on, the study of specific rates of energy decay was well
developed by Schonbek \cite{S}, \cite{S2}, Kajikiya and Miyakawa \cite{KM},
Wiegner \cite{W} and Maremonti \cite{Mar}.
In particular, in \cite{W}  it was established that if the initial disturbance
$w(0)\in L^2_\sigma$ satisfies 
that the underlying Stokes flow decays like
$\|e^{t\Delta}w(0)\|_2=O(t^{-\alpha})$
as $t\to\infty$ for some $\alpha >0$,
then every weak solution $w(t)$ to \eqref{NS-0} (with $g=0$) 
satisfying the strong energy inequality \eqref{SEI-0}
enjoys the decay property
\begin{equation}
\|w(t)\|_2=O(t^{-\beta}), \qquad 
\beta=\min\left\{\frac{n}{4}+\frac{1}{2},\;\alpha\right\}
\label{optimal}
\end{equation}
as $t\to\infty$.
Note that there is no uniform decay rate for $w(0)\in L^2_\sigma$,
see Schonbek \cite{S2} and Hishida \cite{H}.
The decay rate $t^{-(\frac{n}{4}+\frac{1}{2})}$
is actually optimal, no matter how fast $e^{t\Delta}w(0)$ decays,
unless special structures   such as symmetry are assumed.
In fact, 
estimate from below
\begin{equation}
\|w(t)\|_2\geq Ct^{-(\frac{n}{4}+\frac{1}{2})}
\label{below}
\end{equation}
for large $t$ was studied by 
Schonbek \cite{S2}, \cite{S4} and Miyakawa and Schonbek \cite{MS}.
They gave a necessary and sufficient condition for \eqref{below}
via asymptotic expansion of weak solutions
deduced by Fujigaki and Miyakawa \cite{FM},
from which one can find the leading term for $t\to \infty$.
There  also is some literature on the case where $V$ is a
stationary solution,
see Miyakawa and Sohr \cite[Section 5]{MSo},
Borchers and Miyakawa \cite[Section 4]{BM3}, \cite[Section 4]{BM4}, 
Chen \cite{C}, Miyakawa \cite{Mi97}, \cite{Mi98},
Bjorland and Schonbek \cite{BjS},
Karch and Pilarczyk \cite{KP} 
and the references therein (\cite{MSo}, \cite{BM3}, \cite{BM4} and \cite{C}
studied the energy stability of exterior flows).

Given a basic Navier-Stokes flow $\{V,\Pi\}$ satisfying \eqref{NS-0}, 
we perturb it by a solenoidal vector $u_0 \in L^2_{\sigma}$.
The disturbance will be  denoted  by $\{u,p\}$ and we set
$w=V+u, \,\pi=\Pi+p$ in \eqref{NS-0} then
$\{u,p\}$ satisfies %
\begin{equation}
\begin{split}
&\partial_t u+u\cdot\nabla u
=\Delta u-V\cdot\nabla u-u\cdot\nabla V-\nabla p, \\
&\mbox{div $u$}=0, \\
&u(\cdot,0)=u_0
\end{split}
\label{NS}
\end{equation}
in $\mathbb R^n\times (0,\infty), n=3, 4$.
Provided  the basic flow
$V$ is small enough in $L^\infty(0,\infty; L^{n,\infty})$, we first obtain,
for every weak solution satisfying the
strong energy inequality \eqref{SEI} associated with \eqref{NS}, the energy decay
\[ 
\|u(t)\|_2\to 0,\;\mbox{as}\; t\to\infty.
\]
This may be regarded as a complement  to \cite{KPS}, in which a weak
solution with energy decay property has been actually constructed,
while  our result does not depend on how the weak solution was  constructed.
We note that weak solution obtained by \cite{KPS} fulfills the strong energy 
inequality \eqref{SEI}.

The main result of this paper   is  the obtention  of   algebraic rates
of energy decay. Specifically we establish 
 that if the underlying linearized flow,
 denoted by $v(t)$,
with the same initial velocity $u_0$ decays like $t^{-\alpha}$
in $L^2$ as $t\to\infty$
for some $\alpha>0$, then every weak solution satisfying  the strong energy
inequality \eqref{SEI} decays at the rate  
\begin{equation}
\|u(t)\|_2=O(t^{-\beta}), \;\;\; 
\beta=\min\left\{\frac{n}{4}+\frac{1}{2}-\varepsilon, \;\alpha\right\}\;\mbox{as}\; 
t\to\infty,
\label{almost}
\end{equation}
where $\varepsilon >0$ is arbitrarily small,
provided 
\[
\|V\|_{L^\infty(0,\infty; L^{n,\infty})} \leq\delta
\]
with some small constant $\delta=\delta(\varepsilon)>0$.
We also find the decay rate of 
$\|u(t)-v(t)\|_2$
to see that the linear part $v(t)$ is the leading term of $u(t)$,
as long as $v(t)$ does not decay faster than $t^{-\alpha}$ with
$\alpha <n/4+1/2$.
In view of \eqref{optimal} for $V=0$, our result
seems to be ``almost'' sharp, however, we cannot get rid of
the term ``almost'' on account of presence of $\varepsilon >0$
arising from technical reasons.
Our result is enough to obtain the typical decay rate
$t^{-\frac{n}{2}(\frac{1}{q}-\frac{1}{2})}$
when $u_0\in L^q\cap L^2_\sigma$ for some $q\in [1,2)$.
This improves Corollary 2.3 of Karch and Pilarczyk \cite{KP},
who derived the decay rate
above for the case $q\in (6/5,2)$ when $V$ is the Landau solution,
that is the homogeneous stationary Navier-Stokes flow 
of degree $-1$ in $\mathbb R^3$.
The main result stated above seems to be new even for the case where 
$V\in L^{n,\infty}$ is the stationary flow.
In this case the  we recall the related result by Miyakawa
\cite[Theorem 7.1]{Mi97}, who proved that 
if the stationary flow $V$ has finite energy with
$\sup |x||V(x)|<(n-2)/2$
and if the initial disturbance $u_0\in L^2_\sigma$ fulfills
$\|e^{t\Delta}u_0\|_2=O(t^{-\alpha})$ as $t\to \infty$ for some $\alpha>0$,
then there exists a weak solution satisfying the same
decay property as in \eqref{optimal}.
Indeed the rate is sharp, but Miyakawa's condition $V\in L^2$
is more  restrictive then  in our case.
 
The proof of our main result relies on the Fourier splitting method,
which can be traced back to Schonbek \cite{S},
combined with analysis of large time behavior of solutions to 
the initial value problem for the linearized equation
\begin{equation}
\begin{split}
&\partial_tu-\Delta u+V\cdot\nabla u+u\cdot\nabla V+\nabla p=0, \\
&\mbox{div $u$}=0, \\
&u(\cdot,s)=f
\end{split}
\label{linear}
\end{equation}
in $\mathbb R^n\times (s,\infty)$,
where $s\geq 0$ is the given initial time.
Since $V$ is time-dependent then system  is nonautonomous,
we have to consider \eqref{linear} with such initial time.
As in \cite {S},
we split the energy of \eqref{NS} into time-dependent low and high
frequency regions in the Fourier side.
The decay rate is determined from the low frequency part
and our argument is based on the integral equation
\eqref{integ} for weak solutions in terms of the evolution operator,
that is, the operator which provides a unique solution to \eqref{linear}.

For the linear analysis  technical difficulties  prevent us from using 
standard methods.
If the basic flow were stationary,
then the decay property of the semigroup generated by the linearized operator
would be deduced from spectral analysis;
however, we don't have enough knowledge about such analysis when
the basic flow is time-dependent
(unless it is a time-periodic flow).
We are thus forced to deal with the term
$V\cdot\nabla u+u\cdot\nabla V$
as perturbation from the heat semigroup
$e^{t\Delta}$ by means of the integral equation \eqref{IE}.
But, no matter what $r\in (\frac{n}{n-1},\infty)$ may be, we see that
\[
\|e^{t\Delta}P\mbox{div $(u\otimes V+V\otimes u)$}\|_{r,\infty}
\leq Ct^{-1}\|u\|_{r,\infty}\|V\|_{n,\infty},
\]
where $P$ denotes the Helmholtz projection and
$\|\cdot\|_{r,\infty}$ is the $L^{r,\infty}$-norm.
This suggests that it is delicate whether the perturbation term
is subordinate to $\Delta u$ even though $V$ is small.
A remarkable observation by Yamazaki \cite{Y} is that
a real interpolation technique recovers the convergence of the integral,
see \eqref{yamazaki},
in the Lorentz space $L^{r,1}$.
Among other Lorentz spaces, his argument works only in
$L^{r,1}$, however, one can use the duality relation
$(L^{r,1})^*=L^{r^\prime,\infty}$,
where $1/r^\prime+1/r=1$.
This enables us to construct a decaying solution of \eqref{linear},
but we are faced with another problem.
If we defined the evolution operator $T(t,s)$ by this solution,
then it would seem to be difficult to determine the adjoint $T(t,s)^*$
directly,
which is needed to justify the integral equation \eqref{integ}
for weak solutions.
To get around this difficulty, 
by the use of the triplet 
$H^1_\sigma\subset L^2_\sigma\subset (H^1_\sigma)^*$,
we first recall
the theory of J. L. Lions \cite{JL} to show the generation
of the evolution operator $T(t,s)$.
This possesses less regularity properties than standard evolution operator
of parabolic type, see Tanabe \cite{Ta},
which cannot be applied to \eqref{linear}
because of lack of H\"older continuity of $V$ in time up to $t=0$.
We have also no information about large time behavior of $T(t,s)$,
but an advantage is that we are able to
analyze similarly a backward perturbed
Stokes system whose solution operator exactly coincides with $T(t,s)^*$.
We next identify the decaying solution mentioned above with $T(t,s)f$ and,
as a consequence, conclude the $L^q$-$L^r$ decay estimate of $T(t,s)$,
see \eqref{L^q-L^r}.
Look at \eqref{optimal} for the case $V=0$,
in which the rate $t^{-n/4}$ reminds us of $L^1$-$L^2$ decay estimate
and the additional $t^{-1/2}$ comes from the divergence structure 
of the nonlinearity
$u\cdot\nabla u=\mbox{div $(u\otimes u)$}$.
This observation suggests the importance of analysis of the
composite operator
$T(t,s)P\mbox{div}$,
see \eqref{evo-div}, to find the decay rate \eqref{almost}.
If we had the $L^1$ estimate
\begin{equation}
\|T(t,s)P\mbox{div $F$}\|_1\leq C(t-s)^{-1/2}\|F\|_1,
\label{div-L^1}
\end{equation}
then unpleasant $\varepsilon >0$ could be removed in \eqref{almost}
by using it with $F=u\otimes u$.
But \eqref{div-L^1} still remains open.
In this paper we deduce the estimate of the operator
$T(t,s)P\mbox{div}$ in $L^q$ only for $q\in (1,\infty)$.

The paper is organized as follows.
The next section gives our main results (three theorems) for both
\eqref{linear} and \eqref{NS}.
We divide the linear analsis into two sections.
In section \ref{evo} we develop the $L^2$ theory of 
the backward perturbed Stokes system
as well as \eqref{linear} to show Theorem \ref{generate-1}.
Section \ref{decay-evo} is devoted to the large time behavior 
of the evolution operator obtained in the previous section
to prove Theorem \ref{linear-main}.
In the final section, after justification of the integral equation
\eqref{integ} for weak solutions, we prove Theorem \ref{en-decay}
on algebraic decay of energy of the Navier-Stokes flow.

\section{Main results}
\label{result}

We first  introduce  the notation.
For $1\leq q\leq \infty$ the usual Lebesgue space is denoted by 
$L^q(\mathbb R^n)$ with norm $\|\cdot\|_q$.
By $H^k(\mathbb R^n)$ we denote the standard $L^2$-Sobolev spaces.
For $1<q<\infty$ and $1\leq r\leq\infty$ the Lorentz space
can be constructed via real interpolation
\[
L^{q,r}(\mathbb R^n)=\left(L^1(\mathbb R^n), L^\infty(\mathbb R^n)\right)_{1-1/q,r}
\]
with norm $\|\cdot\|_{q,r}$.
But it is usually defined using   rearrangement and
average functions.
For details, see Bergh and L\"ofstr\"om \cite{BL}.
The only cases  needed in this paper are  when $r=1,\infty$. We recall that 

\[ L^{q,1}(\mathbb R^n)\subset L^q(\mathbb R^n)\subset L^{q,\infty}(\mathbb R^n)\] 
and the duality relation
\[ L^{q,1}(\mathbb R^n)^*=L^{q^\prime,\infty}(\mathbb R^n),\;\mbox{where}\; 1/q^\prime+1/q=1. \]
It is well known that 
\[ f\in L^{q,\infty}(\mathbb R^n)\;\mbox{ if and only if}\;
\sup_{s>0}\, s\,|\{x\in\mathbb R^n; |f(x)|>s\}|^{1/q}<\infty,
\]
where $|\cdot|$ stands for the Lebesgue measure.
Note that the homogeneous function of degree $-n/q$ belongs to 
$L^{q,\infty}(\mathbb R^n)$ and that
$C_0^\infty(\mathbb R^n)$,
the class of smooth functions with compact support,
is not dense there.
Hereafter, we abbreviate $L^q(\mathbb R^n)$ as $L^q$ and so on.
We also adopt the same symbols for vector and scalar function spaces
if there is no confusion.

Let us introduce the solenoidal function spaces.
By $C_{0,\sigma}^\infty$ we denote the class of all vector fields $f$
which are in $C_0^\infty$ and satisfy $\mbox{div $f$}=0$.
Let $1\leq q<\infty$.
The spaces $L^q_\sigma$ and $H^1_\sigma$ denote the completion 
of the class $C_{0,\sigma}^\infty$ in $L^q$ and $H^1$, respectively.
Then $L^q_\sigma=\{f\in L^q; \mbox{div $f$}=0\}$ and
$H^1_\sigma=L^2_\sigma\cap H^1$.
We introduce the operator 
$P=(\delta_{jk}+R_jR_k)_{1\leq j,k\leq n}$,
where $R_j$ denotes the Riesz transform.
For $1<q<\infty$ the operator $P$ is bounded in $L^q$ and it is
the projection onto $L^q_\sigma$ associated with the Helmholtz decomposition
\[
L^q=L^q_\sigma\oplus \{\nabla p;\, p\in L^q_{loc}\}, \qquad 1<q<\infty.
\]
It is the orthogonal decomposition when $q=2$.
Since we have
$L^{q,r}=\left(L^{q_0}, L^{q_1}\right)_{\theta,r}$
by the reiteration theorem in the interpolation theory
provided that
\begin{equation}
1<q_0<q<q_1<\infty, \qquad
\frac{1-\theta}{q_0}+\frac{\theta}{q_1}=\frac{1}{q}, \qquad
1\leq r\leq \infty,
\label{expo-interpo}
\end{equation}
the operator $P$ is bounded in the Lorentz space $L^{q,r}$, too.
As in \cite[section 5]{BM4}, we define the solenoidal Lorentz space by
$L^{q,r}_\sigma:=P(L^{q,r})$
for $1<q<\infty$ and $1\leq r\leq \infty$.
Then we find
$L^{q,r}_\sigma=\left(L^{q_0}_\sigma, L^{q_1}_\sigma\right)_{\theta,r}$
for the exponents satisfying \eqref{expo-interpo} and the duality relation
$(L^{q,1}_\sigma)^*=L^{q^\prime,\infty}_\sigma$, where
$1/q^\prime+1/q=1$.

Let $E$ be a Banach space with norm $\|\cdot\|_E$.
Then ${\cal L}(E)$ stands for the Banach space of
all bounded linear operators in $E$.
Let $I$ be a interval in $\mathbb R$ and $1\leq r\leq \infty$.
By $L^r(I; E)$ we denote the Banach space
which consists of measurable functions $u$ with values in $E$ such that the norm $\|u\|_{L^r(I; E)}<\infty $, where 
\[
\|u\|_{L^r(I; E)}=
\left\{
\begin{array}{ll}
\displaystyle{ \left(\int_I\|u(t)\|_E^r\,dt\right)^{1/r}}, \qquad &r<\infty, \\
\displaystyle{\mbox{esssup}\;\{\|u(t)\|_E;\, t\in I\}}, &r=\infty.
\end{array}
\right.
\]

We denote by $C$ various positive constants
(indepedent of time and data under consideration)
which may change from line to line.

Throughout this paper the basic flow $V(x,t)$ is assumed to satisfy
\begin{equation}
\begin{split}
&\mbox{div $V$}=0, \quad
V\in L^\infty(0,\infty; L^{n,\infty}) 
\cap C_w([0,\infty); L^{n,\infty})  \\
&\mbox{with small $\displaystyle{\|V\|:=\sup_{t>0}\|V(t)\|_{n,\infty}}$}
\end{split}
\label{basic} 
\end{equation}
where $C_w([0,\infty); L^{n,\infty})$
consists of all weak*-continuous functions
with values in $L^{n,\infty}$.
We intend to mention how small $\|V\|$ should be in each statement.

Let us first consider the linearized equation \eqref{linear}.
For each $t\geq 0$, we define the bilinear form
$a(t;\cdot,\cdot): H^1_\sigma\times H^1_\sigma \to\mathbb R$ by
\begin{equation}
a(t; u,w)
:=\langle \nabla u, \nabla w\rangle
+\langle V(t)\cdot\nabla u, w\rangle 
-\langle V(t)\otimes u, \nabla w\rangle \qquad (u,w\in H^1_\sigma)
\label{bilinear}
\end{equation}
where $\langle\cdot,\cdot\rangle$ stands for various duality pairings
as well as the scalar product in $L^2_\sigma$, in particular, the last term means
\[
\langle V(t)\otimes u, \nabla w\rangle
:=\sum_{j,k=1}^n\int_{\mathbb R^n}V_j(t)u_k\partial_kw_j dx.
\]
The form \eqref{bilinear} with time-independent $V$ was used
in Karch and Pilarczyk \cite{KP} to define the linearized operator.
By \eqref{basic} 
together with the generalized H\"older and Sobolev inequalities
we easily find
\begin{equation}
\begin{split}
|\langle V(t)\otimes u, \nabla w\rangle|
&\leq \|V(t)\otimes u\|_2\|\nabla w\|_2 \\
&\leq C\|V(t)\|_{n,\infty}\|u\|_{2_*,2}\|\nabla w\|_2 \\
&\leq c_0\|V(t)\|_{n,\infty}\|\nabla u\|_2\|\nabla w\|_2
\end{split}
\label{HS}
\end{equation}
where $1/2_*=1/2-1/n$. Obviously $\langle V(t)\cdot\nabla u, w\rangle$
can be estimated in the same way.
Thus
\begin{equation} 
|a(t; u,w)|\leq (1+2c_0\|V\|)\|\nabla u\|_2\|\nabla w\|_2 \qquad
(u, w\in H^1_\sigma).
\label{conti}
\end{equation}
Note that \eqref{conti} does not hold when $n=2$.
By $L(t): H^1_\sigma\to (H^1_\sigma)^*$
we denote the linear oparator
associated with $a(t; \cdot,\cdot)$, namely,
\begin{equation}
a(t; u,w)=\langle L(t)u, w\rangle, \qquad
u,w\in H^1_\sigma.
\label{ope} 
\end{equation}
Then \eqref{linear} is reduced to the Cauchy problem
\begin{equation}
u^\prime(t)+L(t)u(t)=0, \quad t\in (s,\infty); \qquad u(s)=f.
\label{cauchy}
\end{equation}
We start with the following theorem,
that provides a solution to \eqref{cauchy}.

In this paper the  constants 
$\delta_j$ ($j=1,2,\cdots$) describing     the smallness of $\|V\|$,
appear. All of them depend on
the space dimension $n$, for instance,
$\delta_1=\delta_1(n)$ in Theorem \ref{generate-1}. 
To simplify the notation we will not specify this dependence in each case.
\begin{theorem}
Let $n\geq 3$ and assume \eqref{basic}.
There is a constant $\delta_1>0$ such that if $\|V\|\leq \delta_1$,
then there exists a family $\{T(t,s)\}_{t\geq s\geq 0}\subset{\cal L}(L^2_\sigma)$
with the following properties:

\begin{enumerate}
\item
For every $s\geq 0$ and $f\in L^2_\sigma$ 
\begin{equation}
\begin{split}
&T(\cdot,s)f\in L^2(s,T; H^1_\sigma)\cap C([s,\infty); L^2_\sigma), \\
&\partial_tT(\cdot,s)f\in L^2(s,T; (H^1_\sigma)^*), \qquad \forall\, T\in (s,\infty)
\end{split}
\label{cl-1}
\end{equation}
together with
\begin{equation}
\begin{split}
&\partial_tT(t,s)f+L(t)T(t,s)f=0, \qquad 
\mbox{a.e. $t\in (s,\infty)$ in $(H^1_\sigma)^*$}, \\
&\lim_{t\to s}\|T(t,s)f-f\|_2=0
\end{split}
\label{forward}
\end{equation}
is the  unique solution to \eqref{cauchy} within the class \eqref{cl-1}.
Furthermore, we have
\begin{equation}
T(\cdot,s)f\in L^\infty(s,\infty; L^2_\sigma), \qquad
\nabla T(\cdot,s)f\in L^2(s,\infty; L^2)
\label{cl-2}
\end{equation}
with the energy inequality
\begin{equation}
\|T(t,s)f\|_2^2+\int_s^t\|\nabla T(\tau,s)f\|_2^2\,d\tau
\leq \|f\|_2^2
\label{energy}
\end{equation}
for all $t\geq s$.

\item
We have
\begin{equation}
T(t,\tau)T(\tau,s)=T(t,s), \quad 0\leq s\leq \tau\leq t
\label{semi}
\end{equation}
in ${\cal L}(L^2_\sigma)$.
\end{enumerate}
\label{generate-1}
\end{theorem}

We next show that the evolution operator $\{T(t,s)\}_{t\geq s\geq 0}$ on
$L^2_\sigma$ extends to that on $L^q_\sigma$ with decay properties.
\begin{theorem}
Let $n\geq 3$ and assume $V$ satisfies  \eqref{basic}.
\begin{enumerate}
\item
Let $\delta_1$ be  the constant in Theorem \ref{generate-1},  $1<q<\infty$.
There is a constant $\delta_2=\delta_2(q)\in (0,\delta_1]$ such that
if $\|V\|\leq \delta_2$, then
the operator $T(t,s)\in {\cal L}(L^2_\sigma)$
in Theorem \ref{generate-1}
extends to a bounded operator
on $L^q_\sigma$ with the property \eqref{semi} in ${\cal L}(L^q_\sigma)$.
Given $r_0\in (q,\infty)$,
there is a constant 
$\delta_3=\delta_3(q,r_0)\in (0,\delta_2]$ such that if
$\|V\|\leq \delta_3$,
then the operator
$T(t,s)\in {\cal L}(L^q_\sigma)$
defined above satisfies
\begin{equation}
\|T(t,s)f\|_r\leq C(t-s)^{-\frac{n}{2}(\frac{1}{q}-\frac{1}{r})}\|f\|_q
\label{L^q-L^r}
\end{equation}
for all $t\in (s,\infty)$, $r\in [q,r_0)$ and $f\in L^q_\sigma$.

\item
Given $r_0\in (1,\infty)$,
there is a constant
$\delta_4=\delta_4(r_0)\in (0,\delta_1]$ such that if
$\|V\|\leq \delta_4$,
then the operator $T(t,s)\in {\cal L}(L^2_\sigma)$ with $s<t$
extends to a bounded operator
from $L^1_\sigma$ to $L^r_\sigma$ for every $r\in (1,r_0)$
with the following properties:
\begin{equation} 
\|T(t,s)f\|_r\leq C(t-s)^{-\frac{n}{2}(1-\frac{1}{r})}\|f\|_1
\label{L^1-L^r}
\end{equation}
for all $t\in (s,\infty)$, $r\in (1,r_0)$ and $f\in L^1_\sigma$;
\begin{equation}
T(t,\tau)T(\tau,s)f=T(t,s)f \quad\mbox{in $L^r_\sigma$},
\qquad 0\leq s< \tau\leq t
\label{semigroup}
\end{equation}
for $r\in (1,r_0)$ and $f\in L^1_\sigma$.
\end{enumerate}
\label{linear-main}
\end{theorem}

\begin{remark}
It is not clear whether $T(t,s)$ extends to a bounded operator on $L^1_\sigma$.
\label{bdd-L1}
\end{remark}

Let us proceed to the study of decay properties of the Navier-Stokes flow.
Given $u_0 \in L^2_\sigma$,
a vector field $u=u(x,t)$ is called a weak solution to \eqref{NS} if
\begin{equation}
u\in C_w([0,\infty); L^2_\sigma)\cap L^\infty_{loc}([0,\infty); L^2_\sigma),\qquad
\nabla u\in L^2_{loc}([0,\infty); L^2)
\label{class-NS} 
\end{equation}
and if
\begin{equation}
\langle u(t),\varphi(t)\rangle 
+\int_s^t \big[
a(\tau; u,\varphi)
+\langle u\cdot\nabla u, \varphi\rangle
\big] d\tau
=\langle u(s),\varphi(s)\rangle 
+\int_s^t \langle u,\partial_\tau\varphi\rangle  d\tau 
\label{weak-NS}  
\end{equation}
for all $t\geq s\geq 0$ and the test functions of class
\\
\begin{equation}
\begin{split}
&\varphi\in C([0,\infty); L^2_\sigma)\cap L^\infty_{loc}([0,\infty); L^{n,\infty}),\\
&\nabla\varphi\in L^2_{loc}([0,\infty); L^2), \qquad
\partial_\tau \varphi\in L^2_{loc}([0,\infty); (H^1_\sigma)^*),
\end{split}
\label{test} 
\end{equation}
where $a(\tau; u,\varphi)$ is the bilinear form given by \eqref{bilinear}.
Our class of test functions is larger than the standard one.
Note that every term in \eqref{weak-NS}
makes sense by \eqref{conti}, \eqref{class-NS} and \eqref{test}.

In this paper we study decay properties of weak solutions which satisfy
the strong energy inequality of the form
\begin{equation}
\begin{split} 
&\|u(t)\|_2^2+2\int_s^t\|\nabla u\|_2^2d\tau
\leq \|u(s)\|_2^2+2\int_s^t\langle V\otimes u, \nabla u\rangle d\tau \\ 
&\mbox{for a.e. $s\geq 0$, including $s=0$, and all $t\geq s$}.
\end{split} 
\label{SEI}  
\end{equation}

The main result on this issue reads as follows.
\begin{theorem}
Let $n=3,4$ and $u_0\in L^2_\sigma$.
Assume \eqref{basic}. Let $\delta_1$ be  the constant in Theorem \ref{generate-1}

\begin{enumerate}
\item
There is a constant $\delta_5\in (0,\delta_1]$ such that if 
$\|V\|\leq \delta_5$, then every weak solution $u(t)$ with \eqref{SEI} satisfies
\begin{equation}
\lim_{t\to\infty}\|u(t)\|_2=0, 
\label{stability} 
\end{equation}
\begin{equation}
\|u(t)-T(t,0)u_0\|_2=o\big(t^{-\frac{n}{4}+\frac{1}{2}}\big)  \qquad (t\to\infty),
\label{asym1}
\end{equation}
where $T(t,s)$ denotes the evolution operator  constructed in Theorem \ref{generate-1}.

\item
Let $\varepsilon >0$ be arbitrarily small.
There is a constant
$\delta_6=\delta_6(\varepsilon)\in (0,\delta_1]$ such that if
$\|V\|\leq \delta_6$ and if
\begin{equation}
\|T(t,0)u_0\|_2\leq C(1+t)^{-\alpha}
\label{top}
\end{equation}
for some $\alpha>0$, then
every weak solution $u(t)$ with \eqref{SEI} satisfies
\begin{equation} 
\|u(t)\|_2\leq C_\varepsilon (1+t)^{-\beta}, \qquad
\beta=\min\left\{\frac{n}{4}+\frac{1}{2}-\varepsilon, \,\alpha\right\}
\label{rate} 
\end{equation}
and, furthermore, enjoys
\begin{equation} 
\|u(t)-T(t,0)u_0\|_2
\leq
\left\{ 
\begin{array}{ll} 
C_\varepsilon (1+t)^{-\frac{n}{4}+\frac{1}{2}-2(1-\varepsilon)\alpha}, \qquad 
&0<\alpha <1/2, \\
C_\varepsilon (1+t)^{-\frac{n}{4}-\frac{1}{2}+\varepsilon}, &\alpha\geq 1/2.
\end{array} 
\right.
\label{asym2}
\end{equation}
\end{enumerate}
\label{en-decay}
\end{theorem}

Estimate \eqref{asym2} tells us that
$u(t)$ is asymptotically equivalent to $T(t,0)u_0$ for $t\to\infty$
provided that $\|T(t,0)u_0\|_2\geq C(1+t)^{-\alpha}$ as well as
\eqref{top} holds with $\alpha<n/4+1/2$.

\begin{remark}
It is interesting to ask whether \eqref{rate} holds true  if
\eqref{top} is replaced by
\[
\|e^{t\Delta}u_0\|_2\leq C(1+t)^{-\alpha}.
\]
When $0<\alpha <1$, this condition implies \eqref{top},
see Proposition \ref{heat-decay}, and hence the answer is affirmative.
\label{from-heat}
\end{remark}

We combine \eqref{rate} with \eqref{L^q-L^r}, \eqref{L^1-L^r} and
\eqref{moment-est} below to yield the following corollary immediately.
\begin{corollary}
Let $n=3,4$ and assume \eqref{basic}.
Let $\delta_1$ be the constant in Theorem \ref{generate-1}.

\begin{enumerate}
\item
Let $1\leq q<2$.
There is a constant $\delta_7=\delta_7(q)\in (0,\delta_1]$ such that
if $\|V\|\leq \delta_7$, then
every weak solution $u(t)$ with \eqref{SEI} satisfies
\[ 
\|u(t)\|_2\leq C(1+t)^{-\frac{n}{2}(\frac{1}{q}-\frac{1}{2})} 
\]
provided $u_0\in L^q\cap L^2_\sigma$.

\item
Given $\varepsilon >0$ arbitrarily small,
there is a constant $\delta_8=\delta_8(\varepsilon)\in (0,\delta_1]$ 
such that if 
$\|V\|\leq \delta_8$, then
every weak solution $u(t)$ with \eqref{SEI} satisfies
\[ 
\|u(t)\|_2\leq C(1+t)^{-\frac{n}{4}-\frac{1}{2}+\varepsilon} 
\]
provided $u_0\in L^1\cap L^2_\sigma$ with
$\int |y||u_0(y)|dy<\infty$.
\end{enumerate}
\label{cor}
\end{corollary}

\section{Evolution operator}
\label{evo}

This section is devoted to the study of solutions to  \eqref{cauchy}
and  the backward perturbed Stokes system 
\eqref{backward} below, specifycally, the adjoint 
system of \eqref{cauchy}.
We thus introduce the adjoint form
$a^*(t; u,w):=a(t; w,u)$
and the associated linear operator
$M(t): H^1_\sigma\to (H^1_\sigma)^*$ by
\begin{equation}
a^*(t; u,w)=\langle w, M(t)u\rangle, \qquad 
u, w\in H^1_\sigma.
\label{ope2}
\end{equation}
Since
\begin{equation}
\langle L(t)u, w\rangle=a(t; u,w)
=a^*(t; w,u)=\langle u, M(t)w\rangle
\label{dual-0}
\end{equation}
for $u, w\in H^1_\sigma$,
we see that $L(t)\subset M(t)^*$ and $M(t)\subset L(t)^*$,
which together with the domains 
$D(L(t))=D(M(t))=H^1_\sigma$ implies
\begin{equation}
L(t)=M(t)^*, \qquad M(t)=L(t)^*.
\label{duality}
\end{equation}

In view of \eqref{ope} and \eqref{ope2}, 
we find from integration by parts
that $L(t)$ and $M(t)$ respectively take the form
\begin{equation}
\begin{split}
&L(t)u=P[-\Delta u+V(t)\cdot\nabla u+\mbox{div $(V(t)\otimes u)$}], \\
&M(t)u=P\big[-\Delta u-V(t)\cdot\nabla u-\sum_{j=1}^n V_j(t)\nabla u_j\big].
\end{split}
\label{concrete}
\end{equation}
From $u\in H^1_\sigma$ and $V(t)\in L^{n,\infty}$
it follows that all the terms
\[ 
\Delta u, \quad
V(t)\cdot\nabla u, \quad
\mbox{div $(V(t)\otimes u)$}, \quad
\sum_{j=1}^n V_j(t)\nabla u_j
\]
belong to
$H^{-1}=(H^1)^*$.
Since $P$ is bounded on $H^1$ and, therefore, on $H^{-1}$,
both representations given by \eqref{concrete} make sense in 
$P(H^{-1})=(H^1_\sigma)^*$.

To show Theorem \ref{generate-1}, we use the unique existence 
theorem due to J.L. Lions \cite[Chap.3, Theorem 1.2]{JL},
which was revisited by Tanabe in his own way,
see \cite[Theorem 5.5.1]{Ta}.
This theorem asks the bilinear form $a(t; u,w)$ to be
merely measurable in $t$ for each $u, w\in H^1_\sigma$.
\bigskip

\noindent
{\it Proof of Theorem \ref{generate-1}}.
In view of \eqref{HS} we take
$\delta_1=\frac{1}{2c_0}$ to find that if
$\|V\|\leq\delta_1$, then
\begin{equation}
|\langle V(t)\otimes u, \nabla u\rangle|
\leq \frac{1}{2}\|\nabla u\|_2^2 \qquad (u\in H^1_\sigma), 
\label{control}
\end{equation}
so that
\begin{equation}
a(t; u,u)
=\|\nabla u\|_2^2-\langle V(t)\otimes u, \nabla u\rangle
\geq \frac{1}{2}\|\nabla u\|_2^2 \qquad (u\in H^1_\sigma).
\label{coercive} 
\end{equation}
Here, we have used
\[
\langle V(t)\cdot\nabla u, u\rangle  
=\frac{1}{2}\int_{\mathbb R^n}\mbox{div $(V(t)|u|^2)$}dx  
=\lim_{R\to\infty}\int_{|x|=R}\frac{x\cdot V(t)}{R}|u|^2 d\sigma=0 
\]
for $u\in H^1_\sigma$.
Hence the Lions theorem mentioned above 
provides a unique solution $u(t)$ of \eqref{cauchy}
for every $s\geq 0$ and $f\in L^2_\sigma$.
We define $T(t,s): L^2_\sigma \to L^2_\sigma$ by
\begin{equation}
T(t,s)f:=u(t), \qquad t\in [s,\infty).
\label{constr1}
\end{equation}
Then it satisfies \eqref{cl-1}, and by uniqueness of solutions
we have the semigroup property \eqref{semi}.
Since $\|u(t)\|_2^2$ is absolutely continuous, we find
from 
$\mbox{\eqref{forward}}_1$, \eqref{ope} and \eqref{coercive} that
\begin{equation}
\begin{split}
\frac{d}{dt}\|u(t)\|_2^2
&=2\langle u^\prime(t), u(t)\rangle
=-2\langle L(t)u(t), u(t)\rangle \\
&=-2\|\nabla u(t)\|_2^2+2\langle V(t)\otimes u(t), \nabla u(t)\rangle \\
&\leq -\|\nabla u(t)\|_2^2
\end{split}
\label{diff-en}
\end{equation}
for a.e. $t\in (s,\infty)$.
We thus obtain \eqref{cl-2} together with the energy inequality \eqref{energy}.
As a consequence,
the family $\{T(t,s)\}_{t\geq s\geq 0}\subset {\cal L}(L^2_\sigma)$
constitutes an evolution operator in $L^2_\sigma$ with
all the properties  required in Theorem \ref{generate-1}.
\hfill
$\Box$
\bigskip

We next consider the backward perturbed Stokes equation
subject to the final condition at $t>0$:
\begin{equation}
-v^\prime(s)+M(s)v(s)=0, \quad s\in [0,t); \qquad v(t)=g.
\label{backward}
\end{equation}

The following proposition provides a solution.
\begin{proposition}
Let $n \geq 3$ and assume \eqref{basic}.
Under the same condition $\|V\|\leq \delta_1$ as in Theorem \ref{generate-1},
there is a family $\{S(t,s)\}_{t\geq s\geq 0}\subset {\cal L}(L^2_\sigma)$
with the following properties:

\begin{enumerate}
\item
For every $t>0$ and $g\in L^2_\sigma$ we have
\begin{equation}
S(t,\cdot)g\in L^2(0,t; H^1_\sigma)\cap C([0,t]; L^2_\sigma), \qquad
\partial_s S(t,\cdot)g\in L^2(0,t; (H^1_\sigma)^*),
\label{cl-adj}
\end{equation}
together with
\begin{equation}
\begin{split}
&-\partial_s S(t,s)g+M(s)S(t,s)g=0, \qquad
\mbox{a.e. $s\in [0,t)$ in $(H^1_\sigma)^*$}, \\
&\lim_{s\to t}\|S(t,s)g-g\|_2=0.
\end{split}
\label{backward-2}
\end{equation}

\item
$\|S(t,s)\|_{{\cal L}(L^2_\sigma)}\leq 1$
\end{enumerate}
\label{generate-2}
\end{proposition}

\begin{proof}
Fix $t>0$.
Since the family $\{M(t-\tau)\}_{\tau\in [0,t]}$ possesses the same
property as that of $\{L(\tau)\}_{\tau\geq 0}$, so that one can apply
the existence theorem due to J.L. Lions
as in Theorem \ref{generate-1}; thus, for every $g\in L^2_\sigma$ 
there exists a unique solution $w$ of class
\[
w\in L^2(0,t; H^1_\sigma)\cap C([0,t]; L^2_\sigma), \qquad
w^\prime \in L^2(0,t; (H^1_\sigma)^*)
\]
to the auxiliary equation
\begin{equation}
w^\prime(\tau)+M(t-\tau)w(\tau)=0, \qquad
\mbox{a.e. $\tau\in (0,t]$ in $(H^1_\sigma)^*$},
\label{auxi}
\end{equation}
subject to the initial condition $w(0)=g$.
Similarly to the solution obtained in Theorem \ref{generate-1},
we have
\[
w\in L^\infty(0,t; L^2_\sigma), \qquad
\nabla w\in L^2(0,t; L^2)
\]
with the energy inequality
\begin{equation}
\|w(\tau)\|_2^2+\int_0^\tau\|\nabla w\|_2^2 d\sigma\leq \|g\|_2^2
\label{ener-auxi}
\end{equation}
for every $\tau\in (0,t]$.
Define $S(t,s): L^2_\sigma\to L^2_\sigma$ by
\begin{equation}
S(t,s)g:=w(t-s),\qquad s\in [0,t],
\label{constr}
\end{equation}
which provides the desired solution to \eqref{backward}
since
\[
\partial_s S(t,s)g=-w^\prime(t-s)=M(s)w(t-s)=M(s)S(t,s)g.
\]
This completes the proof of the proposition.
\end{proof}

The following duality relation holds.
\begin{proposition}
Let $n\geq 3$ and assume \eqref{basic}.
Under the same condition
$\|V\|\leq \delta_1$ as in Theorem \ref{generate-1}, we have
\[
T(t,s)=S(t,s)^*, \qquad S(t,s)=T(t,s)^*, \qquad 0\leq s\leq t,
\]
in ${\cal L}(L^2_\sigma)$.
\label{duality2}
\end{proposition}

\begin{proof}
Let $t>s\geq 0$.
By \eqref{cl-1} and \eqref{cl-adj} we find that 
$(s,t)\ni \tau\mapsto \langle T(\tau,s)f, S(t,\tau)g\rangle$
is absolutely continuous for all $f, g\in L^2_\sigma$, 
see \cite[Lemma 5.5.1]{Ta}.
By  \eqref{forward}, \eqref{backward-2} and \eqref{duality} it follows that
that
\begin{equation*}
\begin{split}
&\quad \partial_\tau\langle T(\tau,s)f, S(t,\tau)g\rangle \\
&=\langle \partial_\tau T(\tau,s)f, S(t,\tau)g\rangle
+\langle T(\tau,s)f, \partial_\tau S(t,\tau)g\rangle \\
&=\langle -L(\tau)T(\tau,s)f, S(t,\tau)g\rangle
+\langle T(\tau,s)f, M(\tau)S(t,\tau)g\rangle=0
\end{split}
\end{equation*}
which, integrating from $s+\varepsilon$ to $t-\varepsilon$, yields
\[
\langle T(t-\varepsilon,s)f, S(t,t-\varepsilon)g \rangle
=\langle T(s+\varepsilon,s)f, S(t,s+\varepsilon)g\rangle
\]
for any $\varepsilon >0$.
Since 
$T(\cdot,s)f\in C([s,\infty); L^2_\sigma)$ and
$S(t,\cdot)\in C([0,t]; L^2_\sigma)$,
taking the limit $\varepsilon \to 0$ leads us to
\[
\langle T(t,s)f, g\rangle=\langle f, S(t,s)g\rangle
\]
for all $f, g\in L^2_\sigma$, which concludes the assertion.
\end{proof}

For later use
we give further regularity of $S(t,s)g$ when $g\in H^1_\sigma$.
\begin{proposition}
Let  $\delta_1$ be  the constant in Theorem \ref{generate-1},  
$n\geq 3$ and assume \eqref{basic}. 
Suppose
$\|V\|\leq\delta_9:=\delta_1/2$.
For every $t>0$ and $g\in H^1_\sigma$
the solution $v(s)=S(t,s)g$ to \eqref{backward} obtained in 
Proposition \ref{generate-2} then enjoys
\begin{equation}
v\in L^\infty(0,t; H^1_\sigma)\cap L^2(0,t; H^2), \qquad
v^\prime\in L^2(0,t; L^2_\sigma)
\label{class-adj2}
\end{equation}
and the equation $\mbox{\eqref{backward-2}}_1$ holds in $L^2_\sigma$.
\label{reg}
\end{proposition}

\begin{proof}
In view of the construction \eqref{constr} of $S(t,\cdot)$,
it suffices to verify
\[
w\in L^\infty(0,t; H^1_\sigma)\cap L^2(0,t; H^2), \qquad 
w^\prime\in L^2(0,t; L^2_\sigma)
\]
for the solution to \eqref{auxi} with $w(0)=g$.
When the basic flow $V$ is the trivial state, one can find the result 
in the textbook by Temam \cite[Chap III, Proposition 1.2]{Tem}.
In our case as well we adopt the Galerkin approximation,
see \cite[Section 4]{KPS} for the Navier-Stokes system \eqref{NS},
to go back to the proof of the Lions theorem.

Let us consider formally the energy relation for solutions to \eqref{auxi}, 
supposing enough regularity.
Denote by  $A$  the Stokes operator:
$A=-P\Delta=-\Delta$ with domain $D(A)=H^2\cap L^2_\sigma$.
By \eqref{concrete} we have
\begin{equation*}
\begin{split}
&\quad \frac{1}{2}\,\frac{d}{d\tau}\|\nabla w(\tau)\|_2^2
=\langle w^\prime(\tau), Aw(\tau)\rangle \\
&=-\|Aw(\tau)\|_2^2
+\Big\langle V(t-\tau)\cdot\nabla w(\tau)+\sum_{j=1}^n V_j(t-\tau)\nabla w_j(\tau),
Aw(\tau)\Big\rangle \\
&\leq -\|Aw(\tau)\|_2^2
+2c_0\|V(t-\tau)\|_{n,\infty}\|\nabla^2w(\tau)\|_2\|Aw(\tau)\|_2,
\end{split}
\end{equation*}
where the constant $c_0$ is as in \eqref{HS}.
Since, in the whole space, $\|\nabla^2w\|_2= \|Aw\|_2$,
we have
\begin{equation}
\|\nabla w(\tau)\|_2^2
+\int_0^\tau\|Aw\|_2^2\,d\sigma
\leq \|\nabla g\|_2^2
\label{appro1}
\end{equation}
provided $\|V\|\leq \delta_9=1/(4c_0)=\delta_1/2$.
By  \eqref{auxi} and the estimate of the lower order terms as above
we have
$\|w^\prime(\tau)\|_2\leq C\|Aw(\tau)\|_2$, which combined with
\eqref{appro1} implies

\begin{equation}
\int_0^\tau \|w^\prime\|_2^2\, d\sigma\leq C\|\nabla g\|_2^2.
\label{appro2}
\end{equation}

We now consider  $\{w_k\}$ the Galerkin approximation
which  satisfies \eqref{appro1}, \eqref{appro2} and \eqref{ener-auxi}
for every $\tau\in (0,t]$.
Using $g\in H^1_\sigma$, we find that $\{w_k\}$ is bounded in
$L^\infty(0,t; H^1_\sigma)$ and that
both $\{w_k^\prime\}$ and $\{Aw_k\}$ are bounded in
$L^2(0,t; L^2_\sigma)$.
This concludes the assertion.
\end{proof}
\begin{remark}
It is not clear whether the same assertion as in Proposition \ref{reg}
holds for $T(t,s)$
because the generator $L(t)$ involves $\nabla V(t)$, see \eqref{concrete},
on which we have made no assumption.
We will need Proposition \ref{reg} only for $S(t,s)$
when we take it as a test function of the Navier-Stokes system.
\label{reg-1}
\end{remark}
\begin{remark}
The only thing which we know about the regularity of $S(t,s)$
seems to be Proposition \ref{reg} and this causes the restriction $n\leq 4$
in Proposition \ref{int-rep}.
This is because we have no information 
about the derivatives of $V$.
If suitable assumption on those were imposed,
higher energy estimates would be possible and imply that
$v\in L^\infty(0,t; H^k)$
provided $g\in L^2_\sigma\cap H^k$.
Once we have that,
we can remove the restriction $n\leq 4$ in Proposition \ref{int-rep}
since what we need is 
$v\in L^\infty(0,t; L^n)$
for $g\in C_{0,\sigma}^\infty$,
see the class \eqref{test} of test functions.
\label{reg-2}
\end{remark}

\section{Decay properties of the linearized flow}
\label{decay-evo}

In this section we
consider some decay properties of the evolution operator $T(t,s)$
obtained in Theorem \ref{generate-1}.
Since it seems to be difficult to derive them directly,
we will rewrite \eqref{cauchy} in  the integral form
\begin{equation}
u(t)=e^{(t-s)\Delta}f
-\int_s^t e^{(t-\tau)\Delta}P\,\mbox{div $(u\otimes V+V\otimes u)$}(\tau)\,d\tau
\label{IE}
\end{equation}
and construct a decaying solution to \eqref{IE} by use of its weak form
\begin{equation}
\langle u(t),\psi\rangle
=\langle e^{(t-s)\Delta}f,\psi\rangle
+\int_s^t \langle (u\otimes V+V\otimes u)(\tau), 
\nabla e^{(t-\tau)\Delta}\psi\rangle\,d\tau,
\qquad \psi\in C_{0,\sigma}^\infty.
\label{w-IE}
\end{equation}
This argument is totally independent of the construction of the 
evolution operator $T(t,s)$ 
in the previous section and it works even for $n=2$.
We will show that both solutions coincide.
There might be another way that we could define the evolution operator
as the unique solution of \eqref{IE} itself,
however, in this case the duality relation given by Proposition \ref{duality2}
is not clear.
We do need Proposition \ref{duality2} to get the description
of weak solutions to the Navier-Stokes system in terms of
the evolution operator.

In order to analyze \eqref{IE}, one needs several estimates of the heat semigroup
$e^{t\Delta}$, which are summarized in the following lemma.
Among them, the remarkable estimate \eqref{yamazaki} below was 
established by Yamazaki \cite{Y} and plays a crucial role.
Note that 
$\|\nabla e^{t\Delta}f\|_{r,1}\leq Ct^{-1}\|f\|_{q,1}$
under the condition \eqref{pair}.
\begin{lemma}
Let $n\geq 2$.

\begin{enumerate}
\item
Let $1< q\leq r<\infty$ and $j=0,1$.
Then
\begin{equation}
\|\nabla^j e^{t\Delta}f\|_{r,1}
\leq Ct^{-\frac{n}{2}(\frac{1}{q}-\frac{1}{r})-\frac{j}{2}}\|f\|_{q,1}
\label{standard}
\end{equation}
for all $t>0$ and $f\in L^{q,1}$.

\item
Let $1\leq q\leq r\leq\infty$.
Then the composite operator $e^{t\Delta}P\mbox{\emph{div}}$
extends to a bounded operator from $L^q$ to $L^r$ with
\begin{equation} 
\|e^{t\Delta}P\mbox{\emph{div} $F$}\|_r
\leq Ct^{-\frac{n}{2}(\frac{1}{q}-\frac{1}{r})-\frac{1}{2}}\|F\|_q
\label{div-est}
\end{equation}
for all $t>0$ and $F\in L^q$.

\item
Let $1<q\leq r<\infty$.
Then the composite operator $e^{t\Delta}P\mbox{\emph{div}}$
extends to a bounded operator from $L^{q,\infty}$ to $L^{r,\infty}$ with
\begin{equation}
\|e^{t\Delta}P\mbox{\emph{div} $F$}\|_{r,\infty}
\leq Ct^{-\frac{n}{2}(\frac{1}{q}-\frac{1}{r})-\frac{1}{2}}\|F\|_{q,\infty}
\label{div-vari}
\end{equation}
for all $t>0$ and $F\in L^{q,\infty}$.
If in particular $1<q<r<\infty$, then
\begin{equation} 
\|e^{t\Delta}P\mbox{\emph{div} $F$}\|_r
\leq Ct^{-\frac{n}{2}(\frac{1}{q}-\frac{1}{r})-\frac{1}{2}}\|F\|_{q,\infty},
\label{div-vari2}
\end{equation}
\begin{equation}
\|e^{(t+h)\Delta}P\mbox{\emph{div} $F$}
-e^{t\Delta}P\mbox{\emph{div} $F$}\|_r
\leq C t^{-\frac{n}{2}(\frac{1}{q}-\frac{1}{r})-\frac{1}{2}-\theta}
h^\theta\|F\|_{q,\infty}
\label{div-diff}
\end{equation}
for all $t>0$, $h>0$ and $F\in L^{q,\infty}$,
where $0<\theta <1$.

\item
(\cite{Mi1})
Let $1\leq r\leq \infty$.
Then
\begin{equation}
\|e^{t\Delta}f\|_r
\leq Ct^{-\frac{n}{2}(1-\frac{1}{r})}(1+t)^{-\frac{1}{2}}
\int_{\mathbb R^n}(1+|y|)|f(y)|\,dy
\label{moment}
\end{equation}
for all $t>0$ and $f\in L^1_\sigma$ with 
$\int |y||f(y)|dy<\infty$.

\item
(\cite{Y})
Suppose
\begin{equation}
1<q<r<\infty, \qquad
\frac{1}{q}-\frac{1}{r}=\frac{1}{n}.
\label{pair}
\end{equation}
Then
\begin{equation}
\int_0^\infty
\|\nabla e^{t\Delta}f\|_{r,1}\, dt\leq C\|f\|_{q,1} 
\label{yamazaki}
\end{equation}
for all $f\in L^{q,1}$.
\end{enumerate}
\label{heat}
\end{lemma}

\begin{proof}
Although all estimates are well known, we will give the proof 
for completeness.
Since
\[
e^{t\Delta}f=G(\cdot,t)*f, \qquad
G(x,t)=(4\pi t)^{-n/2}e^{-\frac{|x|^2}{4t}},
\]
the $L^q$-$L^r$ estimate ($1\leq q\leq r\leq\infty$)
\begin{equation}
\|\nabla^j e^{t\Delta}f\|_r
\leq Ct^{-\frac{n}{2}(\frac{1}{q}-\frac{1}{r})-\frac{j}{2}}\|f\|_q
\label{usual}
\end{equation}
follows from the Young inequality.
Inequality \eqref{standard}  follows by real interpolation.

To show \eqref{div-est}, it suffices to prove
\begin{equation}
\|e^{t\Delta}P\mbox{div $F$}\|_q
\leq Ct^{-1/2}\|F\|_q
\label{div-est2}
\end{equation}
for $q\in [1,\infty]$.
As was done by Miyakawa \cite[Section 2]{Mi2},
we have
\[
(e^{t\Delta}P\mbox{div $F$})_j=\sum_{k,l=1}^n\Gamma_{jkl}(\cdot,t)*F_{kl} \qquad
(1\leq j\leq n)
\]
with the kernel function of the form
\[
\Gamma_{jkl}(x,t)
=\partial_l G(x,t)\delta_{jk}
+\int_t^\infty \partial_j\partial_k\partial_l G(x,s)\,ds 
\]
where $\partial_l=\frac{\partial}{\partial x_l}$.
We then find
\[
\|\Gamma_{jkl}(t)\|_1
\leq Ct^{-1/2}+C\int_t^\infty s^{-3/2}\,ds=Ct^{-1/2}
\]
which yields \eqref{div-est2}.
Both \eqref{div-vari} and \eqref{div-vari2} follow from interpolation.
Since
\begin{equation*}
\begin{split}
e^{(t+h)\Delta}P\mbox{div $F$}-e^{t\Delta}P\mbox{div $F$}
&=\int_t^{t+h}\frac{d}{ds}\,e^{s\Delta} P\mbox{div $F$}\,ds \\
&=\int_t^{t+h}\Delta e^{s\Delta/2}e^{s\Delta/2}P\mbox{div $F$}\,ds,
\end{split}
\end{equation*}
estimate \eqref{div-vari2} yields
\[\CI=
\|e^{(t+h)\Delta}P\mbox{div $F$}
-e^{t\Delta}P\mbox{div $F$}\|_r
\leq Ct^{-\frac{n}{2}(\frac{1}{q}-\frac{1}{r})-\frac{3}{2}}\,h \,\|F\|_{q,\infty}.
\]
It is obvious that 
\[ \CI \leq
 Ct^{-\frac{n}{2}(\frac{1}{q}-\frac{1}{r})-\frac{1}{2}}\|F\|_{q,\infty}.
\]
Combining the two last inequalities for $\CI$   gives \eqref{div-diff}.

Estimate \eqref{moment} was shown by Miyakawa \cite[Lemma 3.3]{Mi1}.
A key observation is that $f\in L^1$ together with $\mbox{div $f$}=0$
implies $\int f(y)\,dy=0$.
Then we get
\begin{equation*}
\begin{split}
\|e^{t\Delta}f\|_1
&\leq \int_{\mathbb R^n}\int_{\mathbb R^n}
|G(x-y,t)-G(x,t)||f(y)|\,dy\, dx \\
&\leq Ct^{-1/2}\int_{\mathbb R^n}|y||f(y)|\, dy
\end{split}
\end{equation*}
which yields \eqref{moment}.

Following Yamazaki \cite{Y}, we give the proof of \eqref{yamazaki}.
We fix the pair of $q$ and $r$ satisfying \eqref{pair},
and take $q_j$ ($j=0,1$) so that
\[
1<q_0<q<q_1<r, \qquad
\frac{1}{q}=\frac{1-\theta}{q_0}+\frac{\theta}{q_1}.
\]
Then the sublinear operator
\[
f\mapsto \big[t\mapsto \|\nabla e^{t\Delta}f\|_{r,1}\big]
\]
is bounded from $L^{q_j,1}(\mathbb R^n)$ to
$L^{p_j,\infty}(\mathbb R_+)$ on account of \eqref{standard},
where
$\frac{1}{p_j}=\frac{n}{2}(\frac{1}{q_j}-\frac{1}{r})+\frac{1}{2}$.
Applying the real interpolation $(\cdot,\cdot)_{\theta,1}$
leads us to \eqref{yamazaki}.
The proof is complete.
\end{proof}

We now construct  a solution to the integral equation \eqref{IE}, that satisfies 
the  following $L^q$-$L^r$ decay estimates.
\begin{proposition}
Let $n\geq 2$ and assume \eqref{basic}. Then

\begin{enumerate}
\item
Let $1\leq q<\infty$.
There exist $r_1=r_1(q)>q$ and $\delta_{10}=\delta_{10}(q)>0$
such that if
$\|V\|\leq \delta_{10}$,
then the equation \eqref{IE} admits a unique solution $u(t)$ 
in the class
\[
X:=\{(\cdot\,-s)^\alpha u\in L^\infty(s,\infty; L^{r_1,\infty}_\sigma)\},\quad
\alpha:=\frac{n}{2}\left(\frac{1}{q}-\frac{1}{r_1}\right).
\]
The solution $u(t)$ belongs to 
$C((s,\infty); L^r_\sigma)$
and satisfies the estimate
\begin{equation}
\|u(t)\|_r\leq C(t-s)^{-\frac{n}{2}(\frac{1}{q}-\frac{1}{r})}\|f\|_q
\label{decay-1}
\end{equation}
for all $t\in (s,\infty)$, $r\in [q,r_1)$ and $f\in L^q_\sigma$,
where the case $r=q$ is excluded when $q=1$.

\item
Let $q_1, q_2\in [1,\infty)$.
There is a constant
\[
\delta_{11}=\delta_{11}(q_1,q_2)
\in (0, \min\{\delta_{10}(q_1),\delta_{10}(q_2)\}]
\]
such that if $\|V\|\leq \delta_{11}$, then
$u_1(t)=u_2(t)$
for every $f\in L^{q_1}_\sigma\cap L^{q_2}_\sigma$,
where $u_j(t)$ denotes the corresponding solution to \eqref{IE} obtained above
for $q=q_j$ ($j=1,2$).
\end{enumerate}
\label{asym-1}
\end{proposition}

\begin{proof}
We take $r_1=r_1(q)$ such that
\begin{equation}
\max\{q,n^\prime\}<r_1<\infty, \qquad
\frac{1}{q}-\frac{1}{r_1}<\frac{2}{n}
\label{expo}
\end{equation}
hence $\alpha\in (0,1)$,
where $1/n^\prime +1/n=1$.
Given $u\in X$, 
which is the Banach space equipped with the norm
\[
\|u\|_X:=\mbox{esssup}\,\big\{(t-s)^\alpha\|u(t)\|_{r_1,\infty};\; 
t\in (s,\infty)\big\},
\]
we are going to estimate
\begin{equation} \label{v}
v(t):=
-\int_s^t e^{(t-\tau)\Delta}P\,\mbox{div $(u\otimes V+V\otimes u)$}(\tau)\,d\tau.
\end{equation}
Choose $p$ so that
$1/p:=1/n+1/r_1$; then, $1<p<n$ since $n^\prime<r_1<\infty$.
The duality relation
$(L^{p^\prime,1}_\sigma)^*=L^{p,\infty}_\sigma$, 
the generalized H\"older inequality, \eqref{basic}, \eqref{standard} and
\eqref{yamazaki}yield that
\begin{equation*}
\begin{split}
|\langle v(t), \psi\rangle| 
&\leq\int_s^t\|(u\otimes V+V\otimes u)(\tau)\|_{p,\infty} 
\|\nabla e^{(t-\tau)\Delta}\psi\|_{p^\prime,1}\, d\tau \\
&\leq C\|V\| \|u\|_X\int_s^t (\tau -s)^{-\alpha} 
\|\nabla e^{(t-\tau)\Delta}\psi\|_{p^\prime,1}\, d\tau \\
&\leq C\|V\| \|u\|_X\big\{\int_s^{(t+s)/2} (\tau -s)^{-\alpha}(t-\tau)^{-1}\,d\tau\,
\|\psi\|_{r_1^\prime,1} \\
&\qquad\qquad\qquad 
+\int_{(t+s)/2}^t (\tau -s)^{-\alpha}
\|\nabla e^{(t-\tau)\Delta}\psi\|_{p^\prime,1}\,d\tau\big\} \\
&\leq c_*\|V\| \|u\|_X\, (t-s)^{-\alpha}\|\psi\|_{r_1^\prime,1}
\end{split}
\end{equation*}
for all $\psi\in C_{0,\sigma}^\infty$,
where
$1/p^\prime +1/p=1/r_1^\prime +1/r_1=1$.
This implies that $v(t)\in L^{r_1,\infty}_\sigma$ with
\begin{equation}
\|v(t)\|_{r_1,\infty}\leq c_*\|V\| \|u\|_X\, (t-s)^{-\alpha}.
\label{w-est}
\end{equation}
 Denote by   $\Psi[u](t)$ the RHS of \eqref{IE}, that is,
$\Psi[u](t)=e^{(t-s)\Delta}f+v(t)$.
Then by \eqref{usual} we have $\Psi[u]\in X$ with
\begin{equation*}
\begin{split}
&\|\Psi[u]\|_X\leq C_0\|f\|_q+c_*\|V\|\|u\|_X, \qquad u\in X, \\
&\|\Psi[u]-\Psi[v]\|_X\leq c_*\|V\|\|u-v\|_X, \qquad u, v\in X.
\end{split}
\end{equation*}
Thus, provided $\|V\|\leq\delta_{10}=\delta_{10}(q)=\frac{1}{2c_*}$, 
there exists a fixed point 
$u\in X$ of the mapping $\Psi$. 
Moreover $\|u\|_X\leq 2C_0\|f\|_q$ and hence for a.e. $t\in (s,\infty)$ we have
\begin{equation} 
\|u(t)\|_{r_1,\infty}\leq 2C_0(t-s)^{-\alpha}\|f\|_q.
\label{zuerst}
\end{equation}

We now show  the following uniform bounds   
for all $t\in (s,\infty)$, $q\in (1,\infty)$,
\begin{equation}
\|u(t)\|_q\leq C\|f\|_q,
\label{unif-bdd}
\end{equation}
\begin{equation}
u\in C((s,\infty); L^q_\sigma).
\label{continuity}
\end{equation}
We recall the exponent  $p=(1/n+1/r_1)^{-1}\in (1,n)$ which was used above.
If $q>p$, then \eqref{unif-bdd} is immediate;
in fact,
by using \eqref{div-vari2} together with \eqref{zuerst} we obtain
\begin{equation}
\begin{split}
\|u(t)\|_q
&\leq C\|f\|_q+C\int_s^t (t-\tau)^{-\frac{n}{2}(\frac{1}{p}-\frac{1}{q})-\frac{1}{2}}
\|V(\tau)\|_{n,\infty}\|u(\tau)\|_{r_1,\infty}\, d\tau \\ 
&\leq C(1+\|V\|)\|f\|_q.
\end{split}
\label{argu}
\end{equation}
To show \eqref{continuity}, it suffices to consider $v(t)$ as defined in (\ref{v}).
Let $s< t<t+h$, then we have
\begin{equation*}
\begin{split}
&\quad v(t+h)-v(t) \\
&=\int_s^t \{U(t+h-\tau)-U(t-\tau)\}H(\tau)\,d\tau
+\int_t^{t+h}U(t+h-\tau)H(\tau)\,d\tau  \\
&=:I+J,
\end{split}
\end{equation*}
where
$U(t):=e^{t\Delta}P\mbox{div}$ is the composite operator
given in Lemma \ref{heat} and
$H(t):=-(u\otimes V+V\otimes u)(t)$.
By \eqref{div-diff} we find for every $\theta\in (0,\alpha)$ that
\begin{equation*}
\begin{split}
\|I\|_q 
&\leq C_\theta
\int_s^t (t-\tau)^{-\frac{n}{2}(\frac{1}{p}-\frac{1}{q})-\frac{1}{2}-\theta}\,
(\tau-s)^{-\alpha}\, d\tau\;\cdot h^\theta  \\
&=C_\theta (t-s)^{-\theta}\, h^\theta.
\end{split}
\end{equation*}
We use \eqref{div-vari2} to obtain
\begin{equation*}
\begin{split}
\|J\|_q
&\leq C\int_t^{t+h}(t+h-\tau)^{-\frac{n}{2}(\frac{1}{p}-\frac{1}{q})-\frac{1}{2}}
(\tau -s)^{-\alpha}\,d\tau \\
&\leq C(t-s)^{-\alpha}h^\alpha.
\end{split}
\end{equation*}
Let $s<\frac{s+t}{2}<t+h<t$, and we next consider
\begin{equation*}
\begin{split}
&\quad v(t+h)-v(t) \\
&=\int_s^{t+h}\{U(t+h-\tau)-U(t-\tau)\}H(\tau)\,d\tau
-\int_{t+h}^t U(t-\tau)H(\tau)\,d\tau  \\
&=:\widetilde I+\widetilde J.
\end{split}
\end{equation*}
Since
\[
\widetilde I=\int_{s-h}^t\{U(t-\tau)-U(t-\tau-h)\}H(\tau+h)\,d\tau,
\]
we use \eqref{div-diff} with $(-h)>0$ to obtain
\[
\|\widetilde I\|_q
\leq C_\theta(t+h-s)^{-\theta}(-h)^\theta
\leq C_\theta\left(\frac{t-s}{2}\right)^{-\theta}(-h)^\theta
\]
for every $\theta\in (0,\alpha)$.
We also find
\[
\|\widetilde J\|_q
\leq C(t+h-s)^{-\alpha}(-h)^\alpha
\leq C\left(\frac{t-s}{2}\right)^{-\alpha}(-h)^\alpha
\]
and we are thus led to \eqref{continuity}.
If $q\in (1,p]$, then 
$q_*\in (n^\prime,r_1]$,
where $1/q_*=1/q-1/n$.
From \eqref{expo} it follows that $q_*>p$. 
Thus  take $r_2$ in such a way that
\[
\max\{p, n^\prime\}<r_2<q_*\leq r_1.
\]
Note that when $q=1$, the choice above for  $r_2$  fails.
It is obvious that the same estimate as in \eqref{argu}
leads to
\[
\|u(t)\|_{r_2,\infty}
\leq C(1+\|V\|)(t-s)^{-\frac{n}{2}(\frac{1}{q}-\frac{1}{r_2})}\|f\|_q.
\]
Estimate \eqref{argu} with  $\|u(\tau)\|_{r_1,\infty}$
 replaced by $\|u(\tau)\|_{r_2,\infty}$ yields \eqref{unif-bdd}.
The continuity \eqref{continuity} is proved in the same way as above.

Let $r\in [q,r_1)$. 
Interpolating $L^r$ between $L^{q,\infty} $ and $L^{r_1,\infty} $ 
in combination   with \eqref{zuerst}, \eqref{unif-bdd}
and \eqref{continuity},
yields \eqref{decay-1} for all $t\in (s,\infty)$
as well as
$u\in C((s,\infty); L^r_\sigma)$
when $q\in (1,\infty)$.

For the proof of \eqref{decay-1} with $q=1$,
it suffices to show
\begin{equation}
\|u(t)\|_{r,\infty}
\leq C(t-s)^{-\frac{n}{2}(1-\frac{1}{r})}\|f\|_1
\label{weak-est}
\end{equation}
for every $r\in (1,r_1)$.
When $r\in [p,r_1)$, we have \eqref{weak-est}
along the same way as in \eqref{argu}.
When $r\in (1,p)$, we find $r_*\in (p,r_1)$,
where $1/r_*=1/r-1/n$.
So, estimate \eqref{argu} in which $\|u(\tau)\|_{r_1,\infty}$
is replaced by $\|u(\tau)\|_{r_*,\infty}$ yields \eqref{weak-est}
for $r\in (1,p)$ as well.
We also find
$u\in C((s,\infty); L^{r,\infty}_\sigma)$ for every $r\in (1,r_1)$,
whose proof is the same as that of \eqref{continuity}.
By interpolation we obtain
$u\in C((s,\infty); L^r_\sigma)$ for every $r\in (1,r_1)$.

\medskip 

We now  show  the uniqueness given by the second  assertion  of the proposition.
As the first step, suppose
\[
1\leq q_1< q_2<r_1(q_1)<r_1(q_2),
\]
where $r_1(\cdot)$ is as in \eqref{expo}.
We then take $r$ such that
\[
\max\{q_2, n^\prime\}< r< r_1(q_1).
\]
Assume that
$\|V\|\leq \min\{\delta_{10}(q_1), \delta_{10}(q_2)\}$.
For $f\in L^{q_1}_\sigma\cap L^{q_2}_\sigma$,
we have two solutions $u_j$ ($j=1,2$) satisfying for all $t\in (s,\infty)$
\[
\|u_j(t)\|_{r,\infty}\leq C(t-s)^{-\frac{n}{2}(\frac{1}{q_j}-\frac{1}{r})}
\|f\|_{q_j}.
\]
Then $w(t):=u_1(t)-u_2(t)$ obeys
\begin{equation}
\langle w(t),\psi\rangle
=\int_s^t 
\langle (w\otimes V+V\otimes w)(\tau), \nabla e^{(t-\tau)\Delta}\psi\rangle\,d\tau,
\qquad \psi\in C_{0,\sigma}^\infty.
\label{differ}
\end{equation}
We will show that $w(t)=0$ first for $t\leq s+1$  and then for $t> s+1$.
It is obvious that the quantity
\[
E:=\sup_{t\in (s,s+1]}
(t-s)^{\frac{n}{2}(\frac{1}{q_1}-\frac{1}{r})}\|w(t)\|_{r,\infty}
\]
is finite.
The same deduction, that leads to \eqref{w-est}, 
shows that
$E\leq C\|V\| E$, which imlies $E=0$ 
provided $\|V\|$ is small enough.
We thus obtain 
$u_1(t)=u_2(t)$ for $s\leq t\leq s+1$.
Then 
\[
\langle w(t),\psi\rangle
=\int_{s+1}^t 
\langle (w\otimes V+V\otimes w)(\tau), \nabla e^{(t-\tau)\Delta}\psi\rangle\,d\tau,
\qquad \psi\in C_{0,\sigma}^\infty
\]
for $t>s+1$ and we know that
$u_1, u_2\in L^\infty(s+1,\infty; L^{r,\infty}_\sigma)$.
Since
\[
|\langle w(t),\psi\rangle|
\leq C\int_{s+1}^t\|V(\tau)\|_{n,\infty}\|w(\tau)\|_{r,\infty}
\|\nabla e^{(t-\tau)\Delta}\psi\|_{p^\prime,1}\,d\tau,
\]
where
$1/p=1/n+1/r$ and $1/p^\prime +1/p=1$,
it follows from \eqref{yamazaki} that
\[
\|w(t)\|_{r,\infty}
\leq C\|V\| \|w\|_{L^\infty(s+1,\infty; L^{r,\infty}_\sigma)}
\]
for $t>s+1$.
As a consequence, 
there is a constant $\delta_{11}=\delta_{11}(q_1,q_2)$ such that
if $\|V\|\leq \delta_{11}$, then
$u_1(t)=u_2(t)$ for all $t\geq s$.

Finally, for general case 
$1\leq q_1<q_2<\infty$,
we take finite exponents $p_1, p_2, \cdots,p_m$ such that
\begin{equation*}
\begin{split}
&q_1<p_1<p_2<...<p_m<q_2, \\
&p_1<r_1(q_1), \qquad
p_{j+1}<r_1(p_j)\quad (1\leq j\leq m-1), \qquad
q_2<r_1(p_m).
\end{split}
\end{equation*}
This is actually possible because $r_1(\cdot)$ is determined
by \eqref{expo}.
Let
\[ 
\|V\|\leq \delta_{11}(q_1,q_2)
:=\min\{\delta_{11}(q_1,p_1), \delta_{11}(p_1,p_2), \cdots, \delta_{11}(p_m,q_2)\}.
\]
Then one can repeat the previous consideration to arrive at the conclusion.
The proof is complete.
\end{proof}

When $f\in L^2_\sigma$, let us identify $T(t,s)f$ with the solution
obtained in Proposition \ref{asym-1}. 
\begin{proposition}
Let $n\geq 3$ and assume \eqref{basic}.
There is a constant
$\delta_{12}\in (0, \min\{\delta_1,\delta_{10}(2)\}]$
such that if $\|V\|\leq \delta_{12}$, then
$T(t,s)f=u(t)$ for every $f\in L^2_\sigma$,
where
$T(t,s)$ is the evolution operator
in Theorem \ref{generate-1},
$u(t)$ denotes the solution obtained in Proposition \ref{asym-1} with $q=2$,
and $\delta_1$ (resp. $\delta_{10}$) is the constant in 
Theorem \ref{generate-1} (resp. Proposition \ref{asym-1}).
\label{identify}
\end{proposition}

\begin{proof}
Set $\widetilde u(t)=T(t,s)f$.
It follows from \eqref{forward}, \eqref{ope} and integration by parts that
\begin{equation*}
\begin{split}
\frac{d}{d\tau}\langle e^{(t-\tau)\Delta}\widetilde u(\tau), \psi\rangle
&=-a\big(\tau; \widetilde u(\tau), e^{(t-\tau)\Delta}\psi\big)
-\langle \widetilde u(\tau), \Delta e^{(t-\tau)\Delta}\psi \rangle   \\
&=-\langle V(\tau)\cdot\nabla \widetilde u(\tau), e^{(t-\tau)\Delta}\psi \rangle
+\langle V(\tau)\otimes\widetilde u(\tau), \nabla e^{(t-\tau)\Delta}\psi\rangle
\end{split}
\end{equation*}
for all $\psi\in C_{0,\sigma}^\infty$.
Integration over $(s,t)$ implies that $\widetilde u(t)$
satisfies \eqref{w-IE}.
Set $w(t)=u(t)-\widetilde u(t)$, which fulfills \eqref{differ}.
We know that both $u$ and $\widetilde u$ belong to 
$L^\infty(s,\infty; L^2_\sigma)$,
see \eqref{cl-2} and \eqref{decay-1} with $r=q=2$.
We are going to show that
$w=0$ in $L^\infty(s,\infty; L^{2,\infty}_\sigma)$.
By \eqref{differ} we observe
\[
|\langle w(t), \psi\rangle|
\leq C\int_s^t\|V(\tau)\|_{n,\infty}\|w(\tau)\|_{2,\infty}
\|\nabla e^{(t-\tau)\Delta}\psi\|_{2_*,1}\,d\tau
\]
where $1/2_*=1/2-1/n$.
Applying \eqref{yamazaki} with $q=2$ leads  to
\[
\|w(t)\|_{2,\infty}\leq
C\|V\|\|w\|_{L^\infty(s,\infty; L^{2,\infty}_\sigma)}
\]
hence choosing $\|V\|$ sufficiently small yields the conclusion of the proposition.
\end{proof}

We are now in a position to prove Theorem \ref{linear-main}.
\bigskip

\noindent
{\it Proof of Theorem \ref{linear-main}}.
Given $q\in [1,\infty)$,
suppose
\begin{equation}
\|V\|\leq \delta_2(q):=\min\{\delta_{11}(q,2), \delta_{12}\},
\label{small-1}
\end{equation}
where $\delta_{11}$ and $\delta_{12}$ are the constants in  the
Propositions \ref{asym-1} and \ref{identify}, respectively.
Let  $t\in (s,\infty)$, $r\in [q,r_1)$ 
($r\in (1,r_1)$ when $q=1$) where $r_1=r_1(q)$ is the exponent in Proposition \ref{asym-1}.
Let  $f\in C_{0,\sigma}^\infty$,  then the conclusions in   Propositions  \ref{asym-1} and \ref{identify} imply
\eqref{L^q-L^r} and \eqref{L^1-L^r}

Thus $T(t,s)$ can be extended to a bounded linear operator from $L^q_\sigma$
to $L^r_\sigma$ with $r\in [q,r_1)$
($r\in (1,r_1)$ when $q=1$),
which satisfies \eqref{L^q-L^r} and \eqref{L^1-L^r}
for all $f\in L^q_\sigma$.
It is also verified from \eqref{semi} in $L^2_\sigma$ that
\begin{equation}
\begin{split}
T(t,\tau)T(\tau,s)f=T(t,s)f \quad
&\mbox{in $L^r_\sigma$}, \qquad  0\leq s\leq\tau\leq t,  \\
&(s<t \;\;\mbox{if $r>q>1$}; \quad s<\tau \;\;\mbox{if $r>q=1$})
\end{split}
\label{semi-2}
\end{equation}
with $r\in [q,r_1)$ ($r\in (1,r_1)$ when $q=1$)
for every $f\in L^q_\sigma$, in particular,
we have \eqref{semi} in ${\cal L}(L^q_\sigma)$ unless $q=1$.

Let $r_0\in (q,\infty)$.
When $r_0>r_1(q)$,
we recall \eqref{expo} to take finite exponents $q_1, q_1, \cdots, q_k$ 
such that
\begin{equation*}
\begin{split}
&\max\{q,n^\prime\}<q_1<q_2<\cdots<q_k<r_0,  \\
&q_1<r_1(q), \qquad
q_{j+1}<r_1(q_j)\quad (1\leq j\leq k-1), \qquad
r_0\leq r_1(q_k).
\end{split}
\end{equation*}
Set
\begin{equation}
\begin{split}
&\delta_3=\delta_3(q,r_0)
:=\min\{\delta_2(q), \delta_2(q_1),\cdots, \delta_2(q_k)\},  \\
&\delta_4=\delta_4(r_0):=\delta_3(1,r_0),
\end{split}
\label{small-2}
\end{equation}
and assume $\|V\|\leq \delta_3$ ($\|V\|\leq\delta_4$ when $q=1$).
Then the semigroup property
\eqref{semi-2} combined with the previous consideration yields
both \eqref{L^q-L^r} and \eqref{L^1-L^r}
for desired $r<r_0$.
Let $f\in L^1_\sigma$.
Once we know 
$T(t,s)f\in L^r_\sigma$ for $r\in (1,r_0)$,
we obtain \eqref{semigroup} for such $r$ from \eqref{semi-2} for $r\in (1,r_1)$.
This completes the proof.
\hfill
$\Box$
\bigskip

We supplement the following duality relation.
\begin{proposition}
Let $n\geq 3$, $\frac{2n}{n+2}\leq q\leq 2$ and assume
\eqref{basic}.
Suppose 
$\|V\|\leq\delta_2(q)$, where
$\delta_2$ is the constant in Theorem \ref{linear-main}.
Then
\begin{equation}
\langle T(t,s)f, g\rangle=\langle f, S(t,s)g\rangle 
\label{duality3}
\end{equation}
for all $f\in L^q_\sigma$, $g\in L^2_\sigma$, $t>0$ and
a.e. $s\in [0,t)$.
\label{dual}
\end{proposition}

\begin{proof}
We first observe that the both sides of \eqref{duality3} make sense.
By \eqref{cl-adj} we have
$S(t,s)g\in H^1_\sigma\subset L^2_\sigma\cap L^{2_*}\subset L^{q/(q-1)}$
for a.e. $s\in [0,t)$
since $2\leq q/(q-1)\leq 2_*$, where
$1/2_*=1/2-1/n$.
On the other hand,
under the condition \eqref{small-1}, we have
$T(t,s)f\in L^r_\sigma$
for $r\in [q,r_1)$ with \eqref{L^q-L^r}.
In view of \eqref{expo} and by $\frac{2n}{n+2}\leq q\leq 2$,
we can choose $r_1>2$, so that 
$T(t,s)f\in L^2_\sigma$.

We now take $\{f_k\}\subset C_{0,\sigma}^\infty$ such that
$\displaystyle{\lim_{k\to\infty}\|f_k-f\|_q=0}$.
By Proposition \ref{duality2} we have
\[
\langle T(t,s)f_k, g\rangle=\langle f_k, S(t,s)g\rangle, \qquad
g\in L^2_\sigma,\; k=1,2,...
\]
from which \eqref{duality3} follows by letting $k\to\infty$.
This completes the proof.
\end{proof}

The following assertion is a simple corollary of \eqref{L^q-L^r}.
\begin{proposition}
Let $n\geq 3$, $m>0$ and assume \eqref{basic}.
There is a constant $\delta_{13}=\delta_{13}(m)\in (0,\delta_1]$ such that 
if $\|V\|\leq\delta_{13}$, then 
\begin{equation} 
\lim_{t\to\infty}
\frac{1}{(1+t)^m}\int_0^t m(1+\tau)^{m-1}\|T(\tau,0)f\|_2^2\,d\tau=0
\label{lin}
\end{equation}
\label{lin-decay}
for every $f\in L^2_\sigma$,
where $\delta_1$ is the constant in Theorem \ref{generate-1}.
\end{proposition}

\begin{proof}
Given $m>0$, we take $q\in (1,2)$ satisfying
$\gamma:=n(\frac{1}{q}-\frac{1}{2})\in (0,m)$.
We fix $r_0\in (2,\infty)$ and assume
$\|V\|\leq \delta_{13}(m):=\delta_3(q,r_0)$,
where $\delta_3$ is the constant in Theorem \ref{linear-main}.
Given $f\in L^2_\sigma$ and arbitrary $\varepsilon >0$,
there is $f_\varepsilon\in C^\infty_{0,\sigma}$ such that
$\|f-f_\varepsilon\|_2\leq\varepsilon$.
Since
$\|T(\tau,0)f_\varepsilon\|_2\leq C(1+\tau)^{-\gamma/2}
(\|f_\varepsilon\|_q+\|f_\varepsilon\|_2)$
by \eqref{L^1-L^r}, we find
\[
\frac{1}{(1+t)^m}\int_0^t m(1+\tau)^{m-1}\|T(\tau,0)f\|_2^2 d\tau
\leq C\varepsilon^2
+\frac{Cm(\|f_\varepsilon\|_q+\|f_\varepsilon\|_2)^2}{(m-\gamma)(1+t)^\gamma}
\]
which implies \eqref{lin}.
\end{proof}

The following proposition plays a key role in deduction of energy decay
of the Navier-Stokes flow.
\begin{proposition}
Let $n\geq 3$, $q\in (1,\infty)$ and assume \eqref{basic}.
There is a constant $\delta_{14}=\delta_{14}(q)\in (0,\delta_2(q)]$ 
such that
if $\|V\|\leq\delta_{14}$,
then the composite operator
$T(t,s)P\mbox{\emph{div}}$
extends to a bounded operator on $L^q$ with estimate
\begin{equation} 
\|T(t,s)P\mbox{\emph{div} $F$}\|_q\leq C(t-s)^{-1/2}\|F\|_q
\label{evo-div} 
\end{equation}
for all $t\in (s,\infty)$ and $F\in L^q(\mathbb R^n)^{n\times n}$,
where $\delta_2$ is the constant in Theorem \ref{linear-main}.
\label{evo-est}
\end{proposition}

\begin{proof}
It suffices to show \eqref{evo-div} for
$F\in C_0^\infty(\mathbb R^n)^{n\times n}$.
Consider the integral equation \eqref{IE}
with $f=P\mbox{div $F$}$.
Given $q\in (1,\infty)$, we take $r_2=r_2(q)$ such that
\begin{equation} 
\max\{q, n^\prime\}<r_2<\infty, \qquad
\frac{1}{q}-\frac{1}{r_2}<\frac{1}{n}
\label{expo-div} 
\end{equation}
where $1/n^\prime+1/n=1$.
Note that we cannot take such $r_2$ when $q=1$.
Set
\[
\beta:=\frac{n}{2}\left(\frac{1}{q}-\frac{1}{r_2}\right)+\frac{1}{2}
\in \left(\frac{1}{2}, 1\right).
\]
For $F\in C_0^\infty$ we may regard $f=P\mbox{div $F$}\in L^q_\sigma$,
for which \eqref{IE} admits a solution 
$u(t)=T(t,s)f$ by
Proposition \ref{asym-1} together with Proposition \ref{identify}.
We have also
estimate \eqref{decay-1} for $r\in [q,r_1)$
under the condition \eqref{small-1},
where $r_1$ is the exponent in Proposition \ref{asym-1}.
Taking account of \eqref{expo} and \eqref{expo-div},
we may suppose $r_2\in (q,r_1)$.
Therefore, for each $T\in (s,\infty)$ we find that
\[
E(T):=\sup_{t\in (s,T]}(t-s)^\beta\|u(t)\|_{r_2,\infty}
\]
is finite.
By \eqref{div-est} we have
\[
\|e^{(t-s)\Delta}f\|_{r_2,\infty}=
\|e^{(t-s)\Delta}P\mbox{div $F$}\|_{r_2,\infty}\leq C(t-s)^{-\beta}\|F\|_q
\]
for $t\in (s,\infty)$, and thus
we follow exactly the same way 
as in the proof of \eqref{w-est},
in which we do need $r_2>n^\prime$ from \eqref{expo-div},
to obtain
\[
E(T)\leq C\|F\|_q+C\|V\| E(T)
\]
with some $C>0$ independent of $T$.
Hence, there is a constant $\delta_{15}=\delta_{15}(q)$ such that,
whenever $\|V\|\leq \delta_{15}$, we find
$E(T)\leq C\|F\|_q$ for all $T>s$, which implies that
\[
\|u(t)\|_{r_2,\infty}\leq C(t-s)^{-\beta}\|F\|_q
\]
for all $t\in (s,\infty)$.
Let
$\|V\|\leq \delta_{14}=\delta_{14}(q):=\min\{\delta_2(q), \delta_{15}(q)\}$.
Using the estimate above combined with \eqref{div-est} and \eqref{div-vari2},
we find
\begin{equation*}
\begin{split} 
&\qquad \|u(t)\|_q  \\ 
&\leq C(t-s)^{-1/2}\|F\|_q
+C\int_s^t (t-\tau)^{-\frac{n}{2}(\frac{1}{n}+\frac{1}{r_2}-\frac{1}{q})-
\frac{1}{2}}
\|V(\tau)\|_{n,\infty}\|u(\tau)\|_{r_2,\infty}d\tau \\
&\leq C(t-s)^{-1/2}\|F\|_q
\end{split} 
\end{equation*}
which completes the proof of \eqref{evo-div} for $F\in C_0^\infty$.
\end{proof}

We next derive the decay rate of the evolution operator when 
the initial velocity $f$ fulfills a moment condition as well as $f\in L^1_\sigma$.
\begin{proposition}
Let $n\geq 3$ and assume \eqref{basic}.
Given $r_0\in (1,\infty)$ and $\varepsilon\in (0,1/2)$, there is a constant
$\delta_{16}=\delta_{16}(r_0,\varepsilon)\in (0,\delta_4(r_0)]$
such that if $\|V\|\leq\delta_{16}$,
then the evolution operator $T(t,s)$ in Theorem \ref{linear-main} enjoys
\begin{equation} 
\|T(t,s)f\|_r
\leq C_\varepsilon\, (t-s)^{-\frac{n}{2}(1-\frac{1}{r})}
(1+t-s)^{-\frac{1}{2}+\varepsilon}
\int_{\mathbb R^n}(1+|y|)|f(y)|\,dy
\label{moment-est} 
\end{equation}
for all $t\in (s,\infty)$, $r\in (1,r_0)$ and 
$f\in L^1_\sigma$ with 
$\int |y||f(y)|dy<\infty$,
where $\delta_4$ is the constant in Theorem \ref{linear-main}.
\label{evo-moment}
\end{proposition}

\begin{proof}
Given $r_0\in (1,\infty)$, 
suppose $\|V\|\leq \delta_4(r_0)$.
We already know \eqref{moment-est}
for $t\in (s,s+1)$ by \eqref{L^1-L^r}.
Our task is thus to show \eqref{moment-est} for $t\geq s+1$.
Under the assumptions on $f$, we have \eqref{moment}, which implies
\[
\|e^{(t-s)\Delta}f\|_{r,\infty}
\leq Cm_0\, (t-s)^{-\frac{n}{2}(1-\frac{1}{r})}(1+t-s)^{-\frac{1}{2}}
\qquad (1<r<\infty),
\]
for $t\in (s,\infty)$, where 
$m_0:=\int (1+|y|)|f(y)|\,dy$.
Note that $\frac{n}{2}(1-\frac{1}{r})+\frac{1}{2}>1$ when $r>\frac{n}{n-1}$.
Given
$\varepsilon\in (0,1/2)$,
we take $r_3=r_3(\varepsilon) \in (\frac{n}{n-1},\frac{n}{n-2})$
such that $\frac{n}{2}(1-\frac{1}{r_3})+\frac{1}{2}=1+\frac{\varepsilon}{2}$.
For $T>s+1$ we set
\[
E(T):=\sup_{t\in (s,T]}(t-s)^\gamma \|T(t,s)f\|_{r_3,\infty}
\]
with
\[
\gamma:=1-\frac{\varepsilon}{2}
=\frac{n}{2}\left(1-\frac{1}{r_3}\right)+\frac{1}{2}-\varepsilon.
\]
Set $u(t)=T(t,s)f$, which is the solution to \eqref{IE} 
in Proposition \ref{asym-1}.
We follow the proof of \eqref{w-est},
in which $\|u\|_X$ is replaced by $E(T)$, and use
\[
\|e^{(t-s)\Delta}f\|_{r_3,\infty}\leq Cm_0(t-s)^{-\gamma}
\]
to obtain $E(T)\leq Cm_0$ for every $T>s+1$, that is,
\[
\|u(t)\|_{r_3,\infty}
\leq Cm_0\, (t-s)^{-\gamma}
\]
for all $t\in (s,\infty)$ provided 
$\|V\|\leq\delta_{17}$ for some constant 
$\delta_{17}=\delta_{17}(\varepsilon)>0$.
Following the proof of \eqref{weak-est},
in which we need the condition $r_3<\frac{n}{n-2}$ in the second step,
we obtain
\[
\|u(t)\|_{r,\infty}\leq
Cm_0\, (t-s)^{-\frac{n}{2}(1-\frac{1}{r})-\frac{1}{2}+\varepsilon}
\]
for every $r\in (1,r_3]$ and, thereby,
\begin{equation} 
\|u(t)\|_r
\leq Cm_0\, (t-s)^{-\frac{n}{2}(1-\frac{1}{r})-\frac{1}{2}+\varepsilon}
\label{mo}
\end{equation}
for all $t\in (s,\infty)$ and $r\in (1,r_3)$.
We fix $r_4\in (1,r_3)$ and assume
\[
\|V\|\leq \delta_{16}=\delta_{16}(r_0,\varepsilon)
:=\min\{\delta_4(r_0), \delta_{17}(\varepsilon), \delta_3(r_4,r_0)\},
\]
where $\delta_3$ is the constant in Theorem \ref{generate-1}.
We then use \eqref{semigroup} and combine
\eqref{mo} with $r=r_4$ with \eqref{L^q-L^r} to obtain
\begin{equation*} 
\begin{split} 
\|u(t)\|_r
&=\|T(t,(t+s)/2)u((t+s)/2)\|_r \\
&\leq C(t-s)^{-\frac{n}{2}(\frac{1}{r_4}-\frac{1}{r})}\|u((t+s)/2)\|_{r_4} \\
&\leq Cm_0\, (t-s)^{-\frac{n}{2}(1-\frac{1}{r})-\frac{1}{2}+\varepsilon}
\end{split}
\end{equation*}
for all $t\in (s,\infty)$ and $r\in [r_3,r_0)$ as well.
We thus obtain \eqref{mo} for every $r\in (1,r_0)$,
which leads us to the conclusion.
\end{proof}

Finally, in order to observe Remark \ref{from-heat},
we will show the following proposition, however, 
under the restriction $0<\alpha <1$.
\begin{proposition}
Let $n\geq 3$ and assume \eqref{basic}.
Suppose that $f\in L^2_\sigma$ satisfies
\[
\|e^{t\Delta}f\|_2=O(t^{-\alpha}) \qquad (t\to \infty)
\]
for some $\alpha \in (0,1)$.
There is a constant $\delta_{18}=\delta_{18}(\alpha)\in (0,\delta_{12}]$
such that if
$\|V\|\leq \delta_{18}$,
then
\[
\|T(t,s)f\|_2\leq C(1+t-s)^{-\alpha}, \qquad 0\leq s\leq t,
\]
with the same $\alpha$, 
where $\delta_{12}$ is the constant in Proposition \ref{identify}.
\label{heat-decay}
\end{proposition}

\begin{proof}
Since $\|T(t,s)f\|_2\leq \|f\|_2$,
we will show that
$\|T(t,s)f\|_2\leq C(t-s)^{-\alpha}$.
By the assumption and \eqref{usual} together with interpolation we have
\begin{equation}
\|e^{t\Delta}f\|_{r,\infty}
\leq Ct^{-\frac{n}{2}(\frac{1}{2}-\frac{1}{r})}(1+t)^{-\alpha}
\label{from-L^2}
\end{equation}
for $r\in (2,\infty)$.
Let $\|V\|\leq \delta_{12}$ and set
$u(t)=T(t,s)f$ for fixed $s\geq 0$,
then Proposition \ref{identify} ensures that $u(t)$ is the
solution to \eqref{IE} obtained in Proposition \ref{asym-1}.
We take $r_5=r_5(\alpha)$ such that
\begin{equation}
2<r_5<\infty, \qquad
\frac{1}{2}-\frac{1}{r_5}<\frac{1}{n}\min\{2(1-\alpha),\, 1\}.
\label{expo-r5}
\end{equation}
For $T\in (s,\infty)$ we set
\[
E(T):=\sup_{t\in (s,T]}(t-s)^\mu \|u(t)\|_{r_5,\infty}, \qquad
\mu=\alpha+\frac{n}{2}\left(\frac{1}{2}-\frac{1}{r_5}\right)\in (\alpha, 1).
\]
By following the proof of \eqref{w-est}
and using \eqref{from-L^2}, we are led to
$E(T)\leq C$, independent of $T$, provided
$\|V\|\leq \delta_{19}$
for some constant $\delta_{19}=\delta_{19}(\alpha)>0$.
We thus obtain
\[
\|u(t)\|_{r_5,\infty}\leq C(t-s)^{-\mu}
\]
for all $t\in (s,\infty)$.
Since we have
$p:=(1/n+1/r_5)^{-1}<2$ by \eqref{expo-r5},
we get
\[
\|u(t)\|_2\leq C(t-s)^{-\alpha}
\]
along the same way as in \eqref{argu}.
This concludes the assertion so long as
$\|V\|\leq \delta_{18}=\delta_{18}(\alpha)
:=\min\{\delta_{12},\delta_{19}(\alpha)\}$.
\end{proof}

\section{Energy decay of the Navier-Stokes flow}
\label{Navier-S}

In this section we establish  Theorem \ref{en-decay}.
For this an  integral equation for weak solutions is deduced
in terms of the evolution operator $T(t,s)$. 
Due to  regularity conditions of the basic flow $V$,
the restriction $n\leq 4$ will be needed.
See Remark \ref{reg-2}.
\begin{proposition}
Let $n=3,4$ and assume \eqref{basic}.
Suppose
\[
\|V\|\leq \delta_{20}:=
\min\{\delta_9, \delta_2(n^\prime)\},
\]
where $\delta_2$ and $\delta_9$ are the constants, respectively,
in Theorem \ref{linear-main} and Proposition \ref{reg},
and $1/n^\prime +1/n =1$.
Then every weak solution to \eqref{NS} satisfies
\begin{equation}
u(t)=T(t,s)u(s)-\int_s^t T(t,\tau)P(u\cdot\nabla u)(\tau)\, d\tau
\label{integ}
\end{equation}
in $L^2_\sigma$ for all $t\geq s\geq 0$.
\label{int-rep}
\end{proposition}

\begin{proof}
We note that the integral term of \eqref{integ}
makes sense in $L^{n^\prime}_\sigma$ as the Bochner integral
by \eqref{L^q-L^r} with
$r=q=n^\prime$ and \eqref{class-NS} since
\[
\int_s^t\|T(t,\tau)P(u\cdot\nabla u)(\tau)\|_{n^\prime}\,d\tau
\leq C\int_s^t \|\nabla u(\tau)\|_2^2\,d\tau<\infty. 
\]
The conclusion tells us that this Bochner integral
belongs to $L^2_\sigma$, however,
does nor claim the Bochner
integrability in $L^2_\sigma$.
Let $\psi\in C_{0,\sigma}^\infty$ and fix $t>0$.
We take $\varphi(\cdot,\tau)=S(t,\tau)\psi=T(t,\tau)^*\psi$ in \eqref{weak-NS}.
This is actually possible on account of \eqref{class-adj2} as well as
\eqref{cl-adj} since
$H^1\subset L^n$ for $n\leq 4$.
Combining  \eqref{backward-2}  with \eqref{dual-0} yields
\[
a(\tau; u,S(t,\tau)\psi)
=\langle u, M(\tau)S(t,\tau)\psi\rangle
=\langle u, \partial_\tau S(t,\tau)\psi\rangle.
\]
This implies that
\[
\langle u(t), \psi\rangle
+\int_s^t \langle (u\cdot\nabla u)(\tau), S(t,\tau)\psi\rangle \,d\tau
=\langle u(s), S(t,s)\psi\rangle.
\]
We here recall the duality;
Proposition \ref{duality2} is enough in the RHS,
while we need the relation \eqref{duality3}
with $q=n^\prime$ in the integrand of the LHS.
As a consequence, we find
\[
\langle u(t), \psi\rangle
+\int_s^t \langle T(t,\tau)P(u\cdot\nabla u)(\tau), \psi\rangle \,d\tau
=\langle T(t,s)u(s), \psi\rangle
\]
for all 
$\psi\in C_{0,\sigma}^\infty\subset L^2_\sigma\cap L^n_\sigma$, 
which yields \eqref{integ} in
$L^2_\sigma +L^{n^\prime}_\sigma$;
however, we find that the integral term of the RHS of \eqref{integ}
belongs to $L^2_\sigma$
since the other two terms are in $L^2_\sigma$.
\end{proof}

Having \eqref{integ} at hand,
we can proceed to the final stage.
First of all, under the condition 
$\|V\|\leq\delta_1=\frac{1}{2c_0}$,
it follows from \eqref{SEI} and \eqref{control} that
\begin{equation}
\begin{split}
&\|u(t)\|_2^2+\int_s^t\|\nabla u\|_2^2d\tau
\leq \|u(s)\|_2^2 \\
&\mbox{for a.e. $s\geq 0$, including $s=0$, and all $t\geq s$}. 
\end{split}
\label{SEI2}
\end{equation}
The fundamental idea for the proof of the energy decay
is the Fourier splitting method due to Schonbek \cite{S}.
We split the energy into time-dependent low and high frequency regions:
\[
\|u(t)\|_2^2
=\int_{|\xi|<\rho(t)}|\widehat u(\xi,t)|^2d\xi
+\int_{|\xi|\geq\rho(t)}|\widehat u(\xi,t)|^2d\xi
\]
where 
$\widehat u$ denotes the Fourier transform in space variable and
$\rho(t)$ is a certain positive function.
We control the latter by dissipation as
\[
\int_{|\xi|\geq\rho(t)}|\widehat u(\xi,t)|^2d\xi
\leq\frac{1}{\rho(t)^2}\int_{|\xi|\geq\rho(t)}|\xi|^2|\widehat u(\xi,t)|^2d\xi
\leq\frac{1}{\rho(t)^2}\|\nabla u(t)\|_2^2
\]
so that
\begin{equation}
\rho(t)^2\|u(t)\|_2^2
\leq \rho(t)^2\int_{|\xi|<\rho(t)}|\widehat u(\xi,t)|^2d\xi
+\|\nabla u(t)\|_2^2.
\label{split}
\end{equation}

We note that  the linearized flow does satisfy \eqref{diff-en}, 
which is  inequality  (\ref{SEI2}) in differential form.
But the weak solution $u(t)$ would satisfy such an inequality, that is 
\[
\frac{d}{dt}\|u(t)\|_2^2\leq -\|\nabla u(t)\|_2^2 +\;\mbox{extra terms}.
\]
If the extra terms above would be zero then \eqref{split} with 
$\rho(t)^2=\frac{m}{1+t},\; m>0$, would imply
\begin{equation}
\begin{split}
\frac{d}{dt}[(1+t)^m\|u(t)\|_2^2]
&\leq (1+t)^m\left(\frac{m}{1+t}\|u(t)\|_2^2-\|\nabla u(t)\|_2^2\right) \\
&\leq m(1+t)^{m-1}\int_{|\xi|<\rho(t)}|\widehat u(\xi,t)|^2d\xi.
\end{split}
\label{dif-form}
\end{equation}
We now need  to work out two issues:\\
1. Estimate the right hand side of \eqref{dif-form},\\
2. Justify the energy method
since  we don't have 
the energy inequality in  the form  \eqref{diff-en}.

We start with a lemma on estimates of the nonlinear part
in the time-dependent low frequency region.
\begin{lemma}
Let $n\geq 3$ and assume \eqref{basic}.
Let $u(t)$ be any weak solution to \eqref{NS} with the energy inequality
\eqref{SEI} only for $s=0$.
Set
\begin{equation}
w(t)=-\int_0^t T(t,\tau)P\mbox{\emph{div} $(u\otimes u)$}(\tau) d\tau.
\label{repre-w}
\end{equation}
Then, for every $\theta\in (0,1)$, 
there is a  constant $\delta_{21}=\delta_{21}(\theta)\in (0,\delta_1]$ 
such that if $\|V\|\leq \delta_{21}$, then
\begin{equation}
\begin{split}
&\quad \int_{|\xi|<\sqrt{\frac{m}{1+t}}}|\widehat w(\xi,t)|^2 d\xi  \\
&\leq C_\theta\|u_0\|_2^{2\theta} \left(\frac{m}{1+t}\right)^{n/2-\theta}
\left(\int_0^t (t-\tau)^{-\frac{1}{2-\theta}}\|u(\tau)\|_2^2 d\tau\right)^{2-\theta}
\end{split}
\label{duhamel}
\end{equation}
for all $t\geq 0$ and $m>0$,
where $\delta_1$ is the constant in Theorem \ref{generate-1}.
\label{fourier}
\end{lemma}

\begin{proof}
Given $\theta\in (0,1)$, let $q\in (1,n^\prime)$ such that
$\frac{1}{q}=1-\frac{\theta}{n}$, where
$1/n^\prime +1/n=1$.
Suppose 
$\|V\|\leq\delta_{21}(\theta):=\delta_{14}(q)\leq\delta_1$,
where $\delta_{14}$ is the constant in Proposition \ref{evo-est}.
By \eqref{evo-div} together with
$\|u\otimes u\|_q\leq C\|u\|_2^{2-\theta}\|\nabla u\|_2^\theta$ and 
\eqref{SEI2} for $s=0$
\begin{equation*}
\begin{split}
\|w(t)\|_q
&\leq C\int_0^t (t-\tau)^{-1/2}\|u\|_2^{2-\theta}\|\nabla u\|_2^\theta \,d\tau \\
&\leq C\|u_0\|_2^\theta
\left(\int_0^t (t-\tau)^{-\frac{1}{2-\theta}}\|u\|_2^2 \,d\tau\right)^{1-\theta/2}.
\end{split} 
\end{equation*}

This together with
\begin{equation*}
\begin{split} 
\int_{|\xi|< \sqrt{\frac{m}{1+t}}}|\widehat w(\xi,t)|^2 d\xi 
&\leq \left(\int_{|\xi|< \sqrt{\frac{m}{1+t}}} d\xi\right)^{1-2/q^\prime}
\|\widehat w(t)\|_{q^\prime}^2  \\ 
&\leq C\left(\frac{m}{1+t}\right)^{n(1/q-1/2)}\|w(t)\|_q^2
\end{split}
\end{equation*}
implies \eqref{duhamel}.
\end{proof}

\begin{remark}
If \eqref{evo-div} were available for $q=1$,
we could take $\theta=0$ in \eqref{duhamel},
which would imply the sharp rate \eqref{rate} with $\varepsilon =0$.
\label{sharp}
\end{remark}

The following lemma justifies the formal energy method explained above
when the weak solution
satisfies the strong energy inequality \eqref{SEI}.
\begin{lemma}
Let $n\geq 3$ and assume \eqref{basic}.
Suppose $\|V\|\leq \delta_1$, where $\delta_1$ is the constant
in Theorem \ref{generate-1}.
Let $u$ be any weak solution to \eqref{NS} satisfying  \eqref{SEI}.
Assume that there is a function
$g\in L^1_{loc}([0,\infty))$ with
$g(t)\geq 0$ and a constant $m>0$ such that
\begin{equation}
\frac{m}{1+t}\|u(t)\|_2^2
\leq g(t) + \|\nabla u(t)\|_2^2 
\label{energy-1} 
\end{equation}
for a.e. $t>0$.
Then we have
\begin{equation} 
(1+t)^m\|u(t)\|_2^2
\leq \|u_0\|_2^2+\int_0^t (1+\tau)^mg(\tau)d\tau
\label{weight}
\end{equation}
for all $t>0$.
\label{justifi}
\end{lemma}

\begin{proof}
We follow the argument of Borchers and Miyakawa \cite{BM}.
By \eqref{SEI2} and \eqref{energy-1} we get
\begin{equation} 
\|u(t)\|_2^2+\int_s^t\frac{m}{1+\tau}\|u(\tau)\|_2^2 \,d\tau
\leq \|u(s)\|_2^2+\int_s^t g(\tau)d\tau.
\label{SEI3} 
\end{equation}
We fix $t>0$ and consider
\[
h(s):=\int_s^t\frac{m}{1+\tau}\|u(\tau)\|_2^2 \,d\tau, \qquad s\in [0,t].
\]
Then \eqref{SEI3} implies that
\begin{equation} 
\begin{split} 
\frac{d}{ds}\{(1+s)^mh(s)\}
&=m(1+s)^{m-1}\{-\|u(s)\|_2^2+h(s)\}  \\ 
&\leq m(1+s)^{m-1}\big\{-\|u(t)\|_2^2+\int_s^tg(\tau)d\tau\big\} 
\end{split}
\end{equation}
for a.e. $s\in [0,t]$.
Integrating from $0$ to $t$, we find
\[
-h(0)\leq
-\{(1+t)^m-1\}\|u(t)\|_2^2+\int_0^t\{(1+s)^m\}^\prime\int_s^t g(\tau)d\tau ds.
\]
An integration by parts gives
\[ 
(1+t)^m\|u(t)\|_2^2 
\leq\|u(t)\|_2^2+h(0)-\int_0^tg(\tau)d\tau  
+\int_0^t (1+s)^m g(s)ds
\]
which together with \eqref{SEI3} for $s=0$ yields \eqref{weight}.
\end{proof}

The proof of \eqref{stability} and \eqref{rate}
will be accomplished by iteration with use of \eqref{ite} below.
Lemmas \ref{fourier} and  \ref{justifi} 
yield the following proposition.
\begin{proposition}
Let $n=3, 4$ and assume \eqref{basic}.
Let $\theta\in (0,1)$ and suppose 
$\|V\|\leq\delta_{22}=\delta_{22}(\theta):=\min\{\delta_{20},\delta_{21}(\theta)\}$,
where $\delta_{20}$ and $\delta_{21}$ 
are the constants in Proposition \ref{int-rep}
and Lemma \ref{fourier}, respectively.
Let $u(t)$ be any weak solution to \eqref{NS} satisfying  \eqref{SEI}.
Then
\begin{equation}
\begin{split}
&\quad (1+t)^m\|u(t)\|_2^2 \\ 
&\leq \|u_0\|_2^2
+2m\int_0^t (1+\tau)^{m-1}\|T(\tau, 0)u_0\|_2^2 \,d\tau
+C_\theta\, m^{1+\frac{n}{2}-\theta}\|u_0\|_2^{2\theta} J(t)
\end{split}
\label{ite}
\end{equation}
for all $t>0$ and $m>0$ with

\begin{equation}
J(t)=\int_0^t (1+\tau)^{m-1-\frac{n}{2}+\theta}
\left(\int_0^\tau (\tau-s)^{-\frac{1}{2-\theta}}\|u(s)\|_2^2\,ds\right)^{2-\theta}d\tau.
\label{J}
\end{equation}
\label{itera}
\end{proposition}

\begin{proof}
Estimate \eqref{split} with $\rho(t)^2=\frac{m}{1+t}$ 
tells us that \eqref{energy-1} holds with
\[ 
g(t)=
\frac{m}{1+t}\int_{|\xi|<\sqrt{\frac{m}{1+t}}}|\widehat u(\xi,t)|^2 d\xi.
\]
By virtue of the integral equation \eqref{integ} with $s=0$,
Lemma \ref{justifi} together Lemma \ref{fourier} implies \eqref{ite}.
\end{proof}

To find the asymptotic behavior \eqref{asym1} and \eqref{asym2},  
following  arguments of Borchers and Miyakawa in \cite{BM},
we derive a strong energy inequality for $u(t)-T(t,0)u_0$. 

\begin{remark}
The case $s=0$ is not covered in the following Lemma, since  the last integral 
in \eqref{SEI4} below is not summable in time up to $\tau=0$.
\end{remark}
\begin{lemma}
Let $n\geq 3$ and assume \eqref{basic}.
Let $u$ be any weak solution to \eqref{NS}  satisfying \eqref{SEI}.
Set $v(t)=T(t,0)u_0$ and $w(t)=u(t)-v(t)$.
There is a constant $\delta_{23}\in (0,\delta_1]$ such that
if $\|V\|\leq \delta_{23}$, then
\begin{equation} 
\|w(t)\|_2^2+\int_s^t\|\nabla w\|_2^2 \,d\tau 
\leq \|w(s)\|_2^2+2\int_s^t\langle u\cdot\nabla u, v \rangle \,d\tau
\label{SEI4}
\end{equation}
for a.e. $s>0$ and all $t\geq s$.
\label{en-w}
\end{lemma}

\begin{proof}
By \eqref{diff-en} we have the energy equality
\[
\|v(t)\|_2^2+2\int_s^t\|\nabla v\|_2^2d\tau
=\|v(s)\|_2^2 
+2\int_s^t\langle V\otimes v, \nabla v \rangle d\tau 
\]
which together with \eqref{SEI} implies
\begin{equation}
\begin{split}
&\quad \|w(t)\|_2^2+2\int_s^t\|\nabla w\|_2^2d\tau-\|w(s)\|_2^2 \\ 
&\leq 2\int_s^t
\left[\langle V\otimes u, \nabla u\rangle +\langle V\otimes v, \nabla v\rangle\right]
d\tau \\ 
&\quad -2\langle u(t), v(t)\rangle +2\langle u(s), v(s)\rangle
-4\int_s^t\langle\nabla u, \nabla v\rangle d\tau
\end{split}
\label{en-w2}
\end{equation}
for a.e. $s\geq 0$, including $s=0$, and all $t\geq s$.
Let $s>0$.
We fix $r_0\in (n,\infty)$ and
assume that
$\|V\|\leq \delta_{23}:=\delta_3(2,r_0)\leq \delta_1$,
where $\delta_3$ is the constant in Theorem \ref{linear-main}.
Since 
by \eqref{L^q-L^r} $\|v(\tau)\|_n\leq C\tau^{-n/4+1/2}\|u_0\|_2$ it follows that
$v\in L^\infty(s,T; L^n)$, 
hence by
\[
|\langle u\cdot\nabla u, v\rangle|\leq C\|\nabla u\|_2^2 \|v\|_n
\]
it follows that
$\int_s^t\langle u\cdot\nabla u, v\rangle d\tau$ converges.
Thus  $v$ can be taken as a test function in the definition 
\eqref{weak-NS} (with $s>0$) of the weak solution $u$.
This is verified from \eqref{weak-NS} for
$u(\cdot+\varepsilon)$ over $[s-\varepsilon, t-\varepsilon]$
with test function $v(\cdot+\varepsilon)$,
where $t\geq s\geq \varepsilon>0$,
since $u(\cdot+\varepsilon)$ is also a weak solution.
We thus have
\begin{equation*}
\begin{split} 
\langle u(t), v(t)\rangle
+\int_s^t [a(\tau; u,v)+\langle u\cdot\nabla u, v\rangle] d\tau
&=\langle u(s), v(s)\rangle 
+\int_s^t \langle u,\partial_\tau v\rangle d\tau \\
&=\langle u(s), v(s)\rangle
-\int_s^t a(\tau; v,u)\,d\tau
\end{split}
\end{equation*}
for all $t\geq s>0$.
Equalities $\mbox{\eqref{forward}}_1$ and \eqref{ope} have been used in the last line.
Since
$\langle V\cdot\nabla u, v\rangle + \langle V\cdot\nabla v, u\rangle =0$,
it follows by the definition \eqref{bilinear} 
of the bilinear form $a(t,\cdot,\cdot)$ that
\[
a(\tau; u,v)+a(\tau; v,u)
=2\langle\nabla u,\nabla v\rangle
-\langle V\otimes u, \nabla v\rangle
-\langle V\otimes v, \nabla u\rangle.
\]
Therefore, we obtain from \eqref{en-w2} that
\[
\|w(t)\|_2^2+2\int_s^t\|\nabla w\|_2^2d\tau-\|w(s)\|_2^2 
= 2\int_s^t [\langle u\cdot\nabla u, v\rangle 
+\langle V\otimes w, \nabla w \rangle] \,d\tau.
\]
Consequently, \eqref{control} implies \eqref{SEI4}.
\end{proof}

We note that the $L^\infty$ estimate 
of $T(t,s)$ is not available in \eqref{L^q-L^r}.
Hence we need a new  estimate of the last term of \eqref{SEI4}.
We will use the following bound with arbitrarily small $\theta$.

\begin{lemma} 
Suppose \eqref{top} holds for some $\alpha\geq 0$ and let  $\theta\in (0,1)$.
Under the assumptions of Lemma \ref{en-w},
there is a constant
$\delta_{24}=\delta_{24}(\theta)\in (0,\delta_{23}]$
independent of $\alpha\geq 0$
such that if $\|V\|\leq\delta_{24}$, then
\[
|\langle u\cdot\nabla u, v\rangle| 
\leq \frac{1}{4}\|\nabla w(t)\|_2^2
+C_\theta\, t^{-n/2+\theta}(1+t)^{-2\alpha}
\|u(t)\|_2^{2(1-\theta)}\|\nabla u(t)\|_2^{2\theta}
\]
for a.e. $t>0$.
\label{last}
\end{lemma}

\begin{proof}
Given $\theta\in (0,1)$, we take
$r\in (2,2_*)$ such that
$1/r=1/2-\theta/n$, where
$1/2_*=1/2-1/n$.
Then we have
\begin{equation*}
\begin{split}
|\langle u\cdot\nabla u, v\rangle| 
&=|\langle u\cdot\nabla w, v\rangle| 
\leq \|u\|_r\|\nabla w\|_2\|v\|_{n/\theta} \\ 
&\leq C\|u\|_2^{1-\theta}\|\nabla u\|_2^\theta\|\nabla w\|_2\|v\|_{n/\theta}.
\end{split}
\end{equation*}
We fix $r_0\in (n/\theta,\infty)$ and assume
$\|V\|\leq \delta_{24}(\theta):=\delta_3(2,r_0)\leq \delta_{23}$, where
$\delta_3$ is the constant in Theorem \ref{linear-main}.
By \eqref{L^q-L^r} and \eqref{top} we have
\[
\|v(t)\|_{n/\theta}\leq Ct^{-\frac{n}{4}+\frac{\theta}{2}}\|T(t/2,0)u_0\|_2
\leq Ct^{-\frac{n}{4}+\frac{\theta}{2}}(1+t)^{-\alpha},
\]
from which the conclusion of the Lemma follows.
\end{proof}

We now establish   a bound  for  $\|w(t)\|_2$   
similar to the one obtained in Lemma \ref{justifi} for $\|u(t)\|_2$ .
\begin{lemma}
Let $n \geq 3$ and assume \eqref{basic}.
Given $\theta\in (0,1)$, suppose
$\|V\|\leq \delta_{24}$,
where $\delta_{24}=\delta_{24}(\theta)$ is the constant in Lemma \ref{last}.
Assume that
$u_0\in L^2_\sigma$ satisfies \eqref{top} for some $\alpha\geq 0$.
Let $u$ be any weak solution to \eqref{NS} with \eqref{SEI} and
set $w(t)=u(t)-T(t,0)u_0$.
Assume further that there are a function $g\in L^1_{loc}((0,\infty))$ with
$g(t)\geq 0$ and a constant $m>0$ such that
\begin{equation} 
\frac{m}{t}\|w(t)\|_2^2
\leq g(t) +\frac{1}{2}\|\nabla w(t)\|_2^2
\label{energy-2}
\end{equation}
for a.e. $t>0$.
Then we have
\begin{equation}
t^m\|w(t)\|_2^2\leq
s^m\|w(s)\|_2^2 
+\int_s^t \tau^m [g(\tau)+g_0(\tau)]d\tau
\label{weight-w}
\end{equation}
for all $t\geq s>0$, where
\[
g_0(\tau)
:=C_\theta\, \tau^{-n/2+\theta}(1+\tau)^{-2\alpha}
\|u(\tau)\|_2^{2(1-\theta)}\|\nabla u(\tau)\|_2^{2\theta}.
\]
\label{justifi-w}
\end{lemma}

\begin{proof}
By the definition of $g_0$ above and the class of  weak solution 
we are considering  it follows  that $g_0\in L^1_{loc}((0,\infty))$.
Combining  \eqref{SEI4} with Lemma \ref{last} yields
\[
\|w(t)\|_2^2+\frac{1}{2}\int_s^t\|\nabla w\|_2^2 d\tau
\leq \|w(s)\|_2^2
+\int_s^tg_0(\tau)\,d\tau
\]
for a.e. $s>0$ and all $t\geq s$.
Fix $t>0$.
The last  estimate  combined with \eqref{energy-2} gives
\[
\|w(t)\|_2^2+h(s)\leq \|w(s)\|_2^2+\int_s^t [g(\tau)+g_0(\tau)]d\tau
\]
for a.e. $s\in (0,t]$, where
\[
h(s)=\int_s^t \frac{m}{\tau}\|w(\tau)\|_2^2 \,d\tau, \qquad s\in (0,t].
\]
The same steps as in the proof of Lemma \ref{justifi},
will lead to \eqref{weight-w}.
\end{proof}

Combining  the Fourier splitting method for $\|w(t)\|_2^2$
and  Lemma \ref{justifi-w} gives the  
following proposition.
\begin{proposition}
Let $n=3, 4$ and assume \eqref{basic}.
Given $\theta\in (0,1)$,
suppose 
$\|V\|\leq \delta_{25}=\delta_{25}(\theta)
:=\min\{\delta_{20}, \delta_{21}(\theta), \delta_{24}(\theta)\}$,
where $\delta_{20}$ and $\delta_{24}$ are the constants
in Proposition \ref{int-rep} and Lemma \ref{last}, respectively.
Assume that
$u_0\in L^2_\sigma$ satisfies \eqref{top} for some $\alpha\geq 0$.
Let $u$ be any weak solution to \eqref{NS} satisfying  \eqref{SEI} and
set $w(t)=u(t)-T(t,0)u_0$.
Then 
\begin{equation}
t^m\|w(t)\|_2^2\leq C_\theta\|u_0\|_2^{2\theta} [J_1(t)+J_2(t)]
\label{nonlinear} 
\end{equation}
for all $t>0$ and $m\geq n/2$, where
\[
J_1(t)
=\int_0^t \tau^{m-1-\frac{n}{2}+\theta}
\Big(\int_0^\tau (\tau-\sigma)^{-\frac{1}{2-\theta}}
\|u(\sigma)\|_2^2\,d\sigma\Big)^{2-\theta}d\tau,
\]
\[
J_2(t)
=\Big(\int_0^t \tau^{\frac{m-n/2+\theta}{1-\theta}}
(1+\tau)^{-\frac{2\alpha}{1-\theta}}
\|u(\tau)\|_2^2 \,d\tau\Big)^{1-\theta}.
\]
\label{est-nonlin}
\end{proposition}

\begin{proof}
For $w(t)$ we will use  estimate  \eqref{split}.
Let  $\rho(t)^2=\frac{2m}{t}$ to obtain  \eqref{energy-2}
with
\[
g(t)=
\frac{m}{t}\int_{|\xi|<\sqrt{\frac{2m}{t}}}|\widehat w(\xi,t)|^2 d\xi.
\]
 By Proposition \ref{int-rep} for $n=3, 4$, provided $\|V\|\leq \delta_{20}$,
we have the representation \eqref{repre-w}  for  $w(t)$ .
We follow the proof of Lemma \ref{fourier}, in which $m/(1+t)$
is replaced by $2m/t$, to obtain
\[
g(t)\leq C\|u_0\|_2^{2\theta}
\left(\frac{2m}{t}\right)^{1+n/2-\theta}
\Big(\int_0^t (t-\sigma)^{-\frac{1}{2-\theta}}
\|u(\sigma)\|_2^2\, d\sigma\Big)^{2-\theta},
\]
so long as
$\|V\|\leq \delta_{21}(\theta)$.
By Lemma \ref{justifi-w} we obtain \eqref{weight-w} with $g(t)$ given above,
in which we can take the limit $s\to 0$ since $m\geq n/2$.
By \eqref{SEI2} with $s=0$ we easily find
\[
\int_0^t\tau^m g_0(\tau)\,d\tau
\leq C\|u_0\|_2^{2 \theta}J_2(t),
\]
which completes the proof.
\end{proof}

We are now in a position to complete the proof of Theorem \ref{en-decay}.
\bigskip

\noindent
{\it Proof of Theorem \ref{en-decay}}.
We first prove \eqref{stability} and \eqref{rate}.
The proof is done by iteration based on \eqref{ite}.
We take, for instance,
$\theta=\theta_0=1/2$ and fix $m>n/2-1$.
Assume that 
$\|V\|\leq
\delta_{26}:=\min\{\delta_{13}(m), \delta_{22}(1/2)\}$,
where $\delta_{13}$ and $\delta_{22}$ are the constants in
Proposition \ref{lin-decay} and Proposition \ref{itera}, respectively.
Starting from $\|u(t)\|_2\leq \|u_0\|_2$ by \eqref{SEI2}, we see
that the function $J(t)$ given by \eqref{J} satisfies
\begin{equation}
J(t)\leq C(1+t)^{m-\frac{n}{2}+1}.
\label{J-1st} 
\end{equation}
Therefore we divide \eqref{ite} by $(1+t)^m$ to conclude
\eqref{stability} on account of \eqref{lin}.

We next fix $m>\max\{n/2+1, 2\alpha\}$.
By \eqref{ite} together with \eqref{top} we obtain
\[
\|u(t)\|_2^2\leq C(1+t)^{-m}+C(1+t)^{-2\alpha}+C(1+t)^{-m}J(t) 
\]
with \eqref{J-1st}.
We thus get \eqref{rate} for $2\alpha\leq n/2-1$.
For the opposite case we use
$\|u(t)\|_2^2\leq C(1+t)^{-n/2+1}$
to improve
\[
J(t)\leq
\left\{
\begin{array}{ll} 
C(1+t)^{m-\frac{3}{2}+\frac{\theta_0}{2}}, & n=3, \\ 
C(1+t)^{m-3+\theta_0}(\log (e+t))^{2-\theta_0}, \qquad & n=4. \\
\end{array} 
\right.
\]
Since we have taken $\theta_0=1/2$, we get \eqref{rate}
for $2\alpha\leq 5/4$ (resp. $2\alpha <5/2$)
when $n=3$ (resp. $n=4$).
Finally, we consider the case where $\alpha$ is large.
Given $\varepsilon >0$ arbitrarily small,
we employ \eqref{ite} with $\theta=2\varepsilon$
and make use of better rate of $L^2$ decay
(so that $\int_0^\infty \|u\|_2^2ds<\infty$)
to improve
\[
J(t)\leq C(1+t)^{m-1-\frac{n}{2}+2\varepsilon}
\]
which concludes \eqref{rate} under the condition
$\|V\|\leq \delta_{22}(2\varepsilon)$.
Note that the estimate of $J(t)$ cannot be improved any more.

We next prove \eqref{asym1} and \eqref{asym2}.
We take, for instance, $\theta=1/2$ again and assume that 
$\|V\|\leq\delta_5:=\min\{\delta_{26}, \delta_{25}(1/2)\}$
to conclude the first assertion of Theorem \ref{en-decay}.
We fix $m\geq n/2$.
Since we have \eqref{top} for $\alpha=0$, we find
\begin{equation*}
\begin{split}
J_1(t)\leq Ct^{m-\frac{n}{2}+1}\,
\Big[\frac{1}{t}\int_0^t \Big(\frac{2}{\tau}
&\int_0^{\tau/2}\|u(\sigma)\|_2^2\,d\sigma\Big)^{2-\theta}d\tau  \\
&+\frac{1}{t}\int_0^t\|u(\tau/2)\|_2^{2(2-\theta)}\,d\tau\Big]
\end{split}
\end{equation*}
and
\[
J_2(t)\leq Ct^{m-\frac{n}{2}+1}
\left(\frac{1}{t}\int_0^t\|u(\tau)\|_2^2\,d\tau\right)^{1-\theta},
\]
from which \eqref{nonlinear} yields
\[
\lim_{t\to\infty}t^{\frac{n}{2}-1}\|w(t)\|_2^2=0
\]
on account of \eqref{stability}.
This proves \eqref{asym1}.

Finally, we fix $\varepsilon_0\in (0,1/2)$ and assume that
$\|V\|\leq \delta_{22}(2\varepsilon_0)$.
We also fix $m\geq n/2+2\alpha+2\beta+1$, where $\beta$ is as in \eqref{rate}.
Since we know \eqref{rate}, we obtain
\[
J_1(t)
\leq C\left\{ 
\begin{array}{ll}
t^{m-\frac{n}{2}+1-4\alpha+2\theta\alpha} & (\alpha <1/2), \\
t^{m-\frac{n}{2}-1+\theta}(\log (e+t))^{2-\theta} \qquad & (\alpha=1/2), \\
t^{m-\frac{n}{2}-1+\theta} & (\alpha >1/2),
\end{array}
\right.
\]
and
\[
J_2(t)\leq Ct^{m-\frac{n}{2}+1-2\alpha-2\beta+2\theta\beta}
\]
for all $t>0$.
When $\alpha \in (0,1/2)$, then $\beta=\alpha$, so that both rates are same.
When $\alpha \geq 1/2$, we see that the decay rate of $J_2(t)$ is 
faster than that of $J_1(t)$.
Given $\varepsilon >0$ arbitrarily small,
we take $\theta= 2\varepsilon$ (resp. $\theta= \varepsilon$)
for $\alpha\neq 1/2$ (resp. $\alpha =1/2$)
and suppose
$\|V\|\leq\delta_{25}(2\varepsilon)$
(resp. $\delta_{25}(\varepsilon)$);
then, we obtain \eqref{asym2} from \eqref{nonlinear}.
Hence, under the condition
\[
\|V\|\leq \delta_6(\varepsilon)
:=\min\{\delta_{22}(2\varepsilon), \delta_{22}(2\varepsilon_0),
\delta_{25}(2\varepsilon), \delta_{25}(\varepsilon)\},
\]
we conclude the second assertion of Theorem \ref{en-decay}.
The proof is complete.
\hfill
$\Box$

\end{document}